\documentclass[12pt, oneside,reqno]{amsart}
\usepackage{hyperref}
\usepackage{geometry}
\usepackage{amsmath}
\usepackage{cite}

\newcommand{\rn}{\mathbb R^n}
\newcommand{\sn}{S^{n-1}}
\newcommand{\kno}{\mathcal K^n_o}
\newcommand{\kne}{\mathcal K^n_e}

\newcommand{\hm}{\mathcal H^{n-1}}

\newcommand\wtilde[1]{\overset{\lower.4ex\hbox{$\scriptstyle \sim$}}{#1}}

\numberwithin{equation}{section}

\newtheorem{thm}{Theorem}[section]
\newtheorem{cor}[thm]{Corollary}
\newtheorem{lem}[thm]{Lemma}

\newtheorem{prop}[thm]{Proposition}
\newtheorem{rem}[thm]{Remark}
\newtheorem{defi}[thm]{Definition}

\begin{document}
	\title[Generalized Gaussian measure  ]{The Generalized Gaussian Minkowski problem }
	\author[J. Liu]{Jiaqian Liu}
	\address{School of Mathematics and statistics, Henan University, Jinming Avenue, 475001, Kaifeng, China}
	\email{liujiaqian@henu.edu.cn}
	\author[S. Tang]{Shengyu Tang}
	\address{Institute of Mathematics, Hunan University,  Lushan S Road, 410082, Changsha, China}
	\email{tsy@hnu.edu.cn}

	
	\thanks{Research of Liu is  supported by  the China National Postdoctoral Program for Innovative Talents of CPSF (BX20240102).}
	
	\maketitle
	\begin{abstract}
		This article delves into the $L_p$ Minkowski problem within the framework of generalized Gaussian probability space. This type of probability space was initially introduced in information theory through the seminal works of Lutwak, Yang, and Zhang \cite{LYZ04,LYZ05}, as well as by Lutwak, Lv, Yang, and Zhang \cite{LLYZ}. The primary focus of this article lies in examining the existence of this problem. While the variational method is employed to explore the necessary and sufficient conditions for the existence of the normalized Minkowski problem when $p \in \mathbb{R} \setminus \{0\}$, our main emphasis is on the existence of the generalized Gaussian Minkowski problem without the normalization requirement, particularly in the smooth category for $p \geq 1$.
	\end{abstract}
	
	\section{Introduction}
	In recent decades, the connection between convex geometry and information theory has garnered significant attention from researchers. Concepts and results in convex geometry are deeply linked to those in information theory. Notably, these include the entropy power inequality, moment-entropy inequality, Fisher information inequality, and others(referenced in \cite{CFGLSW,Lv23,CC84,DD02,DD03,Ha19,LLYZ13,LYZ02,LYZ04,LYZ05,LYZ06,LYZ07}). For example, volume is related to entropy, and the Brunn-Minkowski inequality is similar to the entropy power inequality. While the methods of proof typically diverge, one can often discern the corresponding outcomes of one inequality from the initial results of the other. See in detail in the good reference \cite{MM}.
	
	In the moment-entropy inequality, Shannon initially posited in \cite{S48} that a random variable characterized by a fixed second moment and maximal Shannon entropy must conform to the Gaussian distribution. Stam in \cite{S59} demonstrated that a random variable endowed with a specific Fisher information and minimal Shannon entropy also necessitates the Gaussian distribution. The family of generalized Gaussian distributions was first introduced in the paper \cite{LYZ04} by Lutwak, Yang and Zhang
in 2004. New moment-entropy inequalities were proved and their extremals are generalized Gaussian distribution.
	Subsequently, the authors  introduced novel forms of Fisher information and relative Shannon entropy in \cite{LLYZ} and \cite{LYZ05}. Intriguingly, analogous results were observed in relation to these novel metrics, indicating that they also lead to the Gaussian distribution, albeit of a new kind termed generalized Gaussian distribution.
	
	The Minkowski problem is a fundamental problem in convex geometry. It prescribes the surface area measure that is the differential of the volume functional over convex bodies. In 1938, Aleksandrov introduced the pivotal variational formula \cite{A38}, which states: given two convex bodies $K$ and $L$ with non-empty interiors, there exists a Borel measure $S_K(\cdot)$ on $\sn$ such that
	\begin{equation}\label{classical varitional formula}
		\lim_{t\rightarrow0}\frac{V(K+tL)-V(K)}{t}=\int_{\sn}h_L(v)dS_K(v),
	\end{equation}
	where $\sn$ is the unit sphere in Euclidean space, $V$ is the volume,  $h_L$ is the support function of $L$ and $K+L$ is the Minkowski sum of $K$ and $L$. The surface area measure $S_K(\cdot)$ can also be defined explicitly by the convex body $K$ on $\sn$,
	\[
	S_K(\omega)=\mathcal{H}^{n-1}(\nu_K^{-1}(\omega))
	\]
	for any Borel set $\omega\subset\sn$, where $\mathcal{H}^{n-1}$ denotes the $n-1$-dimension Hausdorff measure and $\nu_K$ the Gauss map of $K$. By this definition, the measure characterization problem is the Minkowski problem: \textit{for a given non-zero finite Borel measure $\mu$ on $\sn$, what are the necessary and sufficient conditions so that there exists a (unique) convex body $K$ such that $S_{K}=\mu$?} This classical Minkowski problem was solved by Minkowski himself for polytope case in \cite{M03}, and the general case was solved respectively by Aleksandrov and Fenchel, Jessen.
	
	In \cite{L93}, Lutwak introduced the $L_p$ Minkowski problem, replaced the classical Minkowski sum to the $L_p$ Minkowski sum: for two convex bodies $K$, $L$, the $L_p$ Minkowski sum is denoted by $K+_pL$, if its support function satisfies $h^p_{K+_pL}=h_K^p+h_L^p$. Thus the analogue variational formula to \eqref{classical varitional formula},
	\begin{equation*}
		\lim_{t\rightarrow0}\frac{V(K+_ptL)-V(K)}{t}=\frac{1}{p}\int_{\sn}h^p_L(v)dS_{p,K}(v),
	\end{equation*}
	and $S_{p,K}$ is called the $L_p$ surface area measure. Notably, when $p=1$, the $L_p$ surface area measure reduces to the classical surface area measure. $S_{p,K}$ corresponds to a Minkowski problem, specifically the $L_p$ Minkowski problem, and Lutwak tackled the case of even data when $p>1$. In a related study, B\"{o}r\"{o}czky-Lutwak-Yang-Zhang delved into the $L_0$ Minkowski problem, also known as the logarithmic Minkowski problem in \cite{Bo13}. They derived a sufficient and necessary condition for the existence in an even data. For further insights into $L_p$ Minkowski problems, see in \cite{BBCY19,CW06,HJX15,LYZ00,Z14,Z15,Z17}. The Orlicz Minkowski problem is a crucial extension of the $L_p$ Minkowski problem, which originated from the work of Haberl-Lutwak-Yang-Zhang \cite{Ha10}. Afterwards, many meaningful jobs emerged, see \cite{HH12,XJL14,ZX14,GHW14,GHWXY19,WXL19,JL19,GHXY20,WXL20} and so on.
	
	In this paper, we study the Minkowski problem in generalized Gaussian probability space. The generalized Gaussian probability volume is given in \cite{LLYZ} as follows
	\begin{defi}\cite{LLYZ}\label{definition}
		For $\alpha>0$ and $q<\frac{\alpha}{n}$, the generalized Gaussian density $g_{\alpha,q}$ is defined by	
		\begin{equation*}
			g_{\alpha,q}(x)=
			\begin{cases}
				\frac{1}{Z(\alpha,q)}[1-\frac{q}{\alpha}|x|^\alpha]_+^{\frac{1}{q}-\frac{n}{\alpha}-1},\qquad &if\ q\neq0\\
				\frac{1}{Z(\alpha,q)}e^{-\frac{1}{\alpha}|x|^\alpha},\qquad &if\ q=0,
			\end{cases}
		\end{equation*}
		where $(x)_+=\max(x,0)$, and $Z(\alpha,q)$ is the partition function such that $g_{\alpha,q}(x)$ integral to 1, i.e., $\int_{\rn}g_{\alpha,q}(x)dx=1$.
		More precisely,
		\begin{equation*}
			Z(\alpha,q)=
			\begin{cases}
				\frac{\frac{\alpha}{n}(\frac{q}{\alpha})^{\frac{n}{\alpha}}\Gamma(\frac{n}{2}+1)}{\pi^{\frac{n}{2}}B(\frac{n}{\alpha},\frac{1}{q}-\frac{n}{\alpha})},\qquad &if\  0<q<\frac{\alpha}{n},\\
				\frac{\Gamma(\frac{n}{2}+1)}{\pi^{\frac{n}{2}}\alpha^{\frac{n}{\alpha}}\Gamma(\frac{n}{\alpha}+1)},\qquad &if\ q=0,\\
				\frac{\frac{\alpha}{n}(\frac{q}{\alpha})^{\frac{n}{\alpha}}\Gamma(\frac{n}{2}+1)}{\pi^{\frac{n}{2}}B(\frac{n}{\alpha},1-\frac{1}{q})},\qquad &if\ q<0,
			\end{cases}
		\end{equation*}
		where $\Gamma(x)=\int_0^\infty e^{-t}t^{x-1}dt$ is the Gamma function
		and $B(x,y)=\int_0^1t^{x-1}(1-t)^{y-1}dt=\int_0^\infty\frac{t^{x-1}}{(1+t)^{x+y}}dt$ is the Beta function.

For a convex body $K$, the generalized Gaussian volume $G_{\alpha,q}(K)$ of $K$ is defined by $$G_{\alpha,q}(K)=\int_{K}g_{\alpha,q}(x)dx.$$
	\end{defi}
	
	When $\alpha=2$ and $q=0$, the generalized Gaussian density takes the form
	$g_{2,0}(x)=\frac{1}{\sqrt{2\pi}^n}e^{-\frac{|x|^2}{2}}$, which is the Gaussian density. And the generalized Gaussian density contains many special cases in information theory, see for example, \cite{O07,K09,N03}.
	
	The variational formula of generalized Gaussian volume utilizes the radial formula in \cite{HLYZ} by Huang-Lutwak-Yang-Zhang, and is given by Hu et al. \cite{H},
	\begin{equation}\label{gaussian surface area measure}
		\lim_{t\rightarrow0}\frac{G_{\alpha,q}(K+tL)-G_{\alpha,q}(K)}{t}=\int_{\sn}h_L(v)dS_{\alpha,q}(K,v)
	\end{equation}
	for $K$ and $L$ containing the origin in the interior, where the generalized Gaussian surface area measure $S_{\alpha,q}(K,\cdot)$ is given by
	\[
	S_{\alpha,q}(K,\eta)=\int_{\nu_K^{-1}(\eta)}g_{\alpha,q}(x)d\mathcal{H}^{n-1}(x).
	\]
	It is natural to propose the generalized Gaussian Minkowski problem.
	
	\textbf{The generalized Gaussian Minkowski problem.} \textit{Fix  $\alpha>0$ and $q<\frac{\alpha}{n}$. For a given non-zero finite Borel measure $\mu$ on $\sn$, what are the necessary and sufficient conditions so that there exists a (unique) convex body $K$ such that
		\begin{equation}\label{measure equation intorduction1}
			S_{\alpha,q}(K,\cdot)=\mu?
		\end{equation}
	}
	It is worth noting that when $\alpha=2$ and $q=0$, the generalized Gaussian density reduces to the Gaussian density. Consequently, the Gaussian Minkowski problem, initially proposed and solved by Huang-Xi-Zhao \cite{H21}, emerges as a special case of the generalized Gaussian Minkowski problem. Following this, Kryvonos-Langharst \cite{KL} solved Minkowski problem for a broad class of Borel measures with density when the given data is even, which includes the generalized Gaussian density. Hu et al. \cite{H} addressed this issue by utilizing Aleksandrov variational solutions prescribed by  $\mu$ with a subspace mass inequality.
	
	When the given measure $\mu$ has a density $f$ with respect to the spherical Lebesgue measure, equation \eqref{measure equation intorduction1} becomes a new kind of Monge-Amp\`{e}re equation,
	\begin{equation}\label{monge ampere in introduction}
		g_{\alpha,q}(|Dh|)\det(h_{ij}+h\delta_{ij})=f,
	\end{equation}
	where $Dh$ is the Euclidean gradient of $h$ in $\rn$.
	If we consider the $L_p$ Minkowski sum, similar to \eqref{gaussian surface area measure}, the variational formula appears: for  $p\neq 0$,
	\begin{equation*}
		\lim_{t\rightarrow0}\frac{G_{\alpha,q}(K+_ptL)-G_{\alpha,q}(K)}{t}=\frac{1}{p}\int_{\sn}h_L(v)^pdS_{p,\alpha,q}(K,v),
	\end{equation*}
	and for $p=0,$
	$$
	\lim_{t\rightarrow0}\frac{G_{\alpha,q}(K+_0tL)-G_{\alpha,q}(K)}{t}=\int_{\sn}h_L(v) dS_{0,\alpha,q}(K,v),
	$$
	where the $L_p$ generalized Gaussian surface area measure
	\[
	S_{p,\alpha,q}(K,\eta)=\int_{\nu_K^{-1}(\eta)}g_{\alpha,q}(x)(x\cdot\nu_K(x))^{1-p}d\mathcal{H}^{n-1}(x).
	\]
	It is also natural to propose the $L_p$ generalized Gaussian Minkowski problem.
	
	\textbf{$L_p$ generalized Gaussian Minkowski problem.} \textit{Fix $p\in\mathbb{R}$, $\alpha>0$ and $q<\frac{\alpha}{n}$. For a given non-zero finite Borel measure $\mu$ on $\sn$, what are the necessary and sufficient conditions so that there exists a (unique) convex body $K$ such that
		\begin{equation}\label{measure equation intorduction}
			S_{p,\alpha,q}(K,\cdot)=\mu?
		\end{equation}
	}
	And its associated Monge-Amp\`ere equation is as follows
	\begin{equation}\label{lp monge ampere in introduction}
			h^{1-p}g_{\alpha,q}(|Dh|)\det(h_{ij}+h\delta_{ij})=f,
	\end{equation}
	which contains the equation \eqref{monge ampere in introduction} as a special case when $p=1$. When $\alpha=2$, $q=0$ with $p\in\mathbb{R}$, i.e., the $L_p$ Gaussian Minkowski problem, has been studied by many researchers in recent years, such as \cite{CHLZ,FLX,L22,f23,SX22}. When $\alpha>0$, $q<\frac{\alpha}{n}$ and $p\geq1$, Kryvonos-Langharst \cite{KL} solved the normalized Minkowski problem with even data.
		
	\subsection{Existence of Weak Solutions to Normalized Minkowski Problem}
	Our first result is the existence of weak solutions to the normalized $L_p$ generalized Gaussian Minkowski problem obtained by variational method and Lagrange multipliers.
	
	When $p>0$, we obtain the following sufficient and necessary conditions to the existence of asymmetric solution to the normalized Minkowski problem through a half-space argument. Instead of the o-symmetric result in \cite{KL} for $p\geq1$, our result is asymmetric under the condition that $G_{\alpha,q}(K)=c\geq\frac{1}{2}$. The reason of this condition is that $G_{\alpha,q}(H^+)=\frac{1}{2}$, $H^+=\{x|x\cdot v\geq 0\} $ and if $K$ has an upper bound and combining the  condition $G_{\alpha,q}(K)=c\geq\frac{1}{2}$, it implies that $K\in\kno$.
	\begin{thm}\label{p>0 existence of N-Minkowski problem}
		For $p>0$, $\alpha>0$ and $q<\frac{\alpha}{n}$, if $\mu$ is a non-zero finite Borel measure on $\sn$, then for every $c\in[\frac{1}{2},1)$ there exists a convex body $K\in\kno$ with $G_{\alpha,q}(K)=c$ such that
		\[
		\frac{S_{p,\alpha,q}(K,\cdot)}{S_{p,\alpha,q}(K,\sn)}=\frac{\mu}{|\mu|}
		\]
		 if and only if $\mu$ is not concentrated on any closed hemispheres.
	\end{thm}
	
	When $p<0$, we establish the sufficient and necessary conditions for the existence of o-symmetric solutions. Our approach is largely inspired by Feng-Hu-Xu\cite{f23}, using a sequence of solutions that vanish on all great subspheres to approximate the solution that does not concentrated on any great subspaces.
	\begin{thm}\label{p<=0 existence of N-Minkowski problem}
		For $\alpha>0$, suppose that $q<0$ with $\frac{\alpha}{q}-\alpha<p<0$ or $0\leq q<\frac{\alpha}{n+\alpha}$ with $p<0$. If $\mu$ is a non-zero even finite Borel measure on $\sn$ , then there exists an o-symmetric convex body $K\in\kne$ such that
		\[
		\frac{S_{p,\alpha,q}(K,\cdot)}{S_{p,\alpha,q}(K,\sn)}=\frac{\mu}{|\mu|}
		\]
		 if and only if $\mu$ is not concentrated on any great subspheres.
	\end{thm}
	
	\subsection{Existence of Weak Solutions to Minkowski Problem}\label{ESSMP}
	
	We establish the existence of smooth solutions to the Monge-Amp\`ere equation \eqref{lp monge ampere in introduction} for $p\geq n$ by employing the continuity method, with uniqueness guaranteed by the maximum principle. In addition, we approximate this smooth solution to the weak solution of $L_p$ generalized Gaussian Minkowski problem. While we can also deduce the existence of solutions with generalized Gaussian volumes greater than $\frac{1}{2}$ using degree theory when $p\geq n$(yielding the same conclusion as Theorem \ref{approximation argument for 1<=p<n}), additional restrictions on $\|f\|_{L_1}$ are required. However, by employing the continuity method here, such limitations on $\|f\|_{L_1}$ can be removed.

	\begin{thm}\label{approximation argument for p>n}
		For $p\geq n$, $\alpha>0$ and $q<\frac{\alpha}{n+\alpha}$, suppose $\mu$ is a finite Borel measure on $\sn$ absolutely continuous with respect to spherical Lebesgue measure $\nu$, i.e., $d\mu=fd\nu$, and $\frac{1}{C}<f<C$ when $p>n$ or $f<\frac{1}{Z(\alpha,q)}$ when $p=n$. Then there exists a convex body $K\in\kno$ such that
		\[
		S_{p,\alpha,q}(K,\cdot)=\mu.
		\]
	\end{thm}	
	
	For $1\leq p<n$, our strategy to establish the existence of smooth solutions to equation \eqref{lp monge ampere in introduction} relies on employing a degree theoretic argument. We believe that the two most crucial yet challenging aspects of this theory in proving the generalized Gaussian Minkowski problem are the uniqueness of the associated isotropic Monge-Ampère equation and the isoperimetric inequality. Chen-Hu-Liu-Zhao use an ODE method to derive the uniqueness of two dimension Gaussian isotropic Monge-Amp\`ere equation in \cite{CHLZ}. In a related work, Ivaki and Milman derived uniqueness results for a class of isotropic curvature problems in \cite{IM23}, which encompass the isotropic Monge-Amp\`ere equation related to the $L_p$ generalized Gaussian Minkowski problem in an o-symmetric setting for $p\geq1$. Ivaki further extended these results in  \cite{I23} beyond the o-symmetric setting, albeit with the requirement that the solution must have a positive lower bound. We will show the lower bound in $C_0$ estimate in Lemma \ref{$C^0$ estimate for 1<p<n}, thus the uniqueness of isotropic Monge-Amp\`ere equation can be established.
	
	To address the challenge posed by the isoperimetric inequality, we approach it from the perspective of the $\log$-convexity or $\log$-concavity of the generalized Gaussian density. A measure $\mu$ in $\rn$ is deemed $\log$-concave or $\log$-convex, if its density $f(x)$ (with respect to the Lebesgue measure) is smooth and $\log f$ is concave or convex. We observe that the generalized Gaussian density  $g_{\alpha,q}$ defined in Definition \ref{definition} is $\log$-concave when $q\in[0,\frac{\alpha}{n+\alpha})$ with $\alpha\geq1$, and $\log$-convex when $q\leq0$ with $\alpha\in(0,1]$ or $q\in[\frac{\alpha}{n+\alpha},\frac{\alpha}{n})$ with $\alpha\geq1$. In \cite{N}, Nolwen provided an isoperimetric inequality for $\log$-concave density. While in \cite{G16}, Gregory provided a proof of $\log$-convex density conjecture, proposed by Kenneth Brakke, which asserts that balls are isoperimetric regions. However, a significant challenge arises when the generalized Gaussian density $g_{\alpha,q}$ is neither log-concave nor log-convex. In \cite{T14}, Takatsu generalized the Poincar\'{e} limit, offering a method to approximate the generalized Gaussian measure $g_{\alpha,q}$ when $q\in(0,\frac{\alpha}{n+\alpha})$ by the projections of the uniform radial probability measure on $\sn$, and derives its isoperimetric inequality. Nonetheless, we highlight that in our article, the isoperimetric inequality is lacking when $q<0$ with $\alpha>1$ or $q\in[\frac{\alpha}{n+\alpha},\frac{\alpha}{n})$ with $\alpha\in(0,1)$, directly implying the absence of existence of smooth solutions to the equation \eqref{lp monge ampere in introduction}.
	
	After solving the aforementioned two aspects, we  establish the existence of smooth solutions for Gaussian volumes both greater than and less than $\frac{1}{2}$ for $1\leq p<n$. By using an appropriate approximation respectively, we derive two weak solutions respectively in Theorem \ref{approximation argument for 1<=p<n} and Theorem \ref{small volume approximation argument for 1<=p<n}. The original idea of deriving solutions with Gaussian volume less than $\frac{1}{2}$ is from Chen-Hu-Liu-Zhao\cite{CHLZ}, who derived it to two dimension Gaussian Minkowski problem.
	
	\begin{thm}\label{approximation argument for 1<=p<n}
		For $1\leq p<n$, suppose that $0\leq q<\frac{\alpha}{n+\alpha}$ with $\alpha>0$ or $q<0$ with $0<\alpha\leq1$. If $\mu$ is a finite Borel measure on $\sn$ and not concentrated on any closed hemisphere with $|\mu|<(\frac{n}{2})^{1-p}I^p[G_{\alpha,q}](\frac{1}{2})$, where $I[G_{\alpha,q}]$ is the isoperimetric profile. Then there exists a convex body $K\in\kno$ with $G_{\alpha,q}(K)>\frac{1}{2}$ such that
		\[
		S_{p,\alpha,q}(K,\cdot)=\mu.
		\]
	\end{thm}
	
	\begin{thm}\label{small volume approximation argument for 1<=p<n}
		For $1\leq p<n$, suppose that $0\leq q<\frac{\alpha}{n+\alpha}$ with $\alpha>0$ or $q<0$ with $0<\alpha\leq1$. If $\mu$ is a finite Borel measure on $\sn$ absolutely continuous with respect to spherical Lebesgue measure $\nu$, i.e., $d\mu=fd\nu$, and $f$ such that $\frac{1}{C}<f<C$ and $\|f\|_{L_1}<(\frac{n}{2})^{1-p}I^p[G_{\alpha,q}](\frac{1}{2})$ where $I[G_{\alpha,q}]$ is the isoperimetric profile. Then there exists a convex body  $K\in\kne$ with $G_{\alpha,q}(K)<\frac{1}{2}$ such that
		\[
		S_{p,\alpha,q}(K,\cdot)=\mu.
		\]
	\end{thm}

	The paper is organized as follows. In section \ref{preliminaries}, we introduce some basic notions in convex geometry. In section \ref{Normalized $L_p$-Generalized Gaussian Minkowski problem}, the existence of weak solutions to the normalized Minkowski problem is proved. In section \ref{B-M section}, Brunn-Minkowski inequality and isoperimetric inequality for generalized Gaussian probability measure are given. In section \ref{continuity method for $p>n$}, we use the continuity method to derive the existence of smooth solution to the equation \eqref{lp monge ampere in introduction} for $p\geq n$, and in section \ref{degree theory}, we use the degree theoretic argument to derive the existence of solutions for $1\leq p<n$. In section \ref{Approximation}, we approximate the smooth solutions which are derived in section \ref{continuity method for $p>n$} and section \ref{degree theory} to the weak solutions of Minkowski problem.
	
	\section{Preliminaries}\label{preliminaries}
	
	Our notations and some basic properties are based on Schneider\cite{Sc14}, which is a good reference to the theory of convex bodies.
	
	The work is on $n$-dimensional real Euclidean vector space $\rn$, with standard scalar product $(\cdot, \cdot)$. for any $x,y\in\rn$, $(x, y)_+$ means the positive part of $(x, y)$. The unit sphere in $\rn$ is denoted by $\sn$ and the unit ball in $\rn$ is denoted by $B$.
	
	Denote by $\kno$ the family of convex bodies contain the origin as an interior, and deonte by $\kne$ the subset of $\kno$ that are o-symmetric.
	
	The Hausdorff metric of two compact convex sets $K$, $L\in\rn$ is defined by
	\[
	d(K,L):=\max\{\sup_{x\in K}\inf_{y\in L}|x-y|,\ \sup_{x\in L} \inf_{y\in K}|x-y|\},
	\]
	or equivalently,
	\[
	d(K,L)=\min\{\lambda\geq0:K\subset L+\lambda B,\ L\subset K+\lambda B\}.
	\]
	The family of compact convex sets in $\sn$ equipped with the Hausdorff metric is local compact, and it is known as the Blaschke's selection theorem. See more details in  \cite{Sc14} Section 1.8.
	
	Suppose $K$ is a compact convex subset in $\rn$. Its support function $h_K$ is defined by
	\[
	h_K(x)=\max\{(x,y):y\in K\}
	\]
	for any $x\in\rn$. When $K$ still contains the origin as its interior, we define
	\[
	\rho_K(x)=\max\{\lambda:\lambda x\in K\}
	\]
	for any $x\in\rn\setminus\{o\}$.
	
	Denote by $C^+(\sn)$ the space of positive continuous functions on $\sn$, and $C^+_e(\sn)$ the space of even positive continuous functions on $\sn$. For $f\in C^+(\sn)$, its associated Wulff shape $[f]$ is a convex body given by
	\[
	[f]=\{x\in\rn:(x\cdot y)\leq f(y)\ \text{for any}\ y\in\sn \}.
	\]
	
	Suppose $K\in\kno$, the polar body $K^*$ of $K$ is defined by
	\[
	K^*=\{x\in\rn:(x,y)\leq1\ \text{for any}\ y\in K\},
	\]
	and the support functions and radial functions about the convex body and its polar body have the following relationship,
	\[
	h_K(x)=\rho^{-1}_{K^*},\ h_{K^*}(x)=\rho^{-1}_K(x)
	\]
	for any $x\in\rn$.
	
	For a convex body $K\in\kno$, the surface area measure $S_K$ is defined by
	\[
	S_K(\omega)=\mathcal{H}^{n-1}(\nu^{-1}_K(\omega))
	\]
	for Borel set $\omega\subset\sn$, where $\mathcal{H}^k$ is $k$-dimension Hausdorff measure, and $\nu_K$ is the Gauss map of $K$.
	
	The $L_p$-generalized Gaussian surface area measure $S_{p,\alpha,q}(K,\cdot)$ weakly converges with respect to the Hausdorff metric, as shown in the following lemma. The proof closely resembles Theorem 3.4 in \cite{H21} which we omit here.
	\begin{lem}\label{weakly convergent wrt Haussdorff metric}
		If $K_i\in\kno$ converges to $K_0\in\kno$ in Hausdorff metric, then $S_{p,\alpha,q}(K_i,\cdot)$ weakly converges to $S_{p,\alpha,q}(K,\cdot)$.
	\end{lem}

	\section{Normalized $L_p$-Generalized Gaussian Minkowski Problem}\label{Normalized $L_p$-Generalized Gaussian Minkowski problem}
	In this section, we use the method in \cite{L22} to obtain the existence of asymmetric solutions for the $L_p$-Generalized Gaussian Minkowski problem for $p>0$, and the method in \cite{f23} to obtain the existence of o-symmetric solutions for $p<0$.
	
	Before proving the existence, we need the variational formula as follows.
	
	\begin{lem}[\bf \cite{L22}]\label{lem2}
		For $p\neq0,$ let $K\in\kno$, and $f: \sn \rightarrow \mathbb R$ be a continuous function. For  small enough $\delta>0$, and each $t\in(-\delta,\delta)$, we define the continuous function $h_t: \sn\rightarrow(0,\infty)$ as
		\begin{equation*}
			h_t(v)=(h_K(v)^p+tf(v)^p)^{\frac{1}{p}},
		\end{equation*}
		where $v\in\sn$. Then,
		\begin{equation*}
			\lim _{t \rightarrow 0} \frac{\rho_{[h_t] }(u)-\rho_K(u)}{t}=\frac{f(\alpha_K(u))^p}{ph_K(\alpha_K(u))^p}\rho_K(u)
		\end{equation*}
		holds for almost all $u\in\sn$. In addition, there exists $M>0$ such that
		\begin{equation*}
			|\rho_{[h_t]}(u)-\rho_K(u)|\leq M|t|,
		\end{equation*}
		for all $u\in\sn$ and $t\in(-\delta,\delta)$.
	\end{lem}
	
	The variational formula of generalized Gaussian is actually a special case of Lemma 2.7 in \cite{KL}.
	\begin{thm}[\bf $L_p$-Variational formula for Generalized Gaussian ]\label{$L_p$-Variational formula for Generalized Gaussian}Let $K\in\kno$, and $f: \sn \rightarrow \mathbb R$ be a continuous function. For  small enough $\delta>0$, and each $t\in(-\delta,\delta)$,
		Then, for  $p\neq 0$
		, one has that
		$$
		\lim _{t \rightarrow 0} \frac{G_{\alpha,q} \left([h_t] \right)-G_{\alpha,q}\left(K\right)}{t}=\frac{1}{p} \int_{\sn}f(v)^p dS_{p,\alpha,q}(K,v),
		$$
		where $h_t(v)=(h_K(v)^p+tf(v)^p)^{\frac{1}{p}}.$
	\end{thm}
	\begin{proof}By the polar coordinates, the generalized Gaussian is presented as
		\begin{equation*}
			G_{\alpha,q}([h_t])=\frac{1}{Z(\alpha,q)}\int_{\sn}\int_{0}^{\rho_{[h_t]}(u)}[1-\frac{q}{\alpha} r^\alpha]_+^{\frac{1}{q}-\frac{n}{\alpha}-1}r^{n-1}drdu.
		\end{equation*}
		Denote $F(t)=\int_{0}^{\rho_{[h_t]}(u)}[1-\frac{q}{\alpha} r^\alpha]_+^{\frac{1}{q}-\frac{n}{\alpha}-1}r^{n-1}dr$, then by Lebesgue differentiate theorem,
		$$
		\lim _{t \rightarrow 0} \frac{F(t) -F(0)}{t}=[1-\frac{q}{\alpha} \rho_K(u)^\alpha]_+^{\frac{1}{q}-\frac{n}{\alpha}-1}\rho_{K}(u)^{n-1}\lim _{t \rightarrow 0} \frac{\rho_{[h_t] }(u)-\rho_K(u)}{t}.
		$$
		Then applying the Lemma \ref{lem2}, dominated convergence  theorem shows that
		for  $p\neq 0$,
		$$
		\lim _{t \rightarrow 0} \frac{G_{\alpha,q} \left([h_t] \right)-G_{\alpha,q}\left(K\right)}{t}=\frac{1}{p} \int_{\sn}f(v)^p dS_{p,\alpha,q}(K,v).
		$$	\end{proof}

	
	The proof is divided into the following three subsections.
	
	\subsection{Associated Maximization Problem}
	In this subsection, we will reduce the existence of solutions of the generalized Gaussian Minkowski problem to the existence of maximizers of its associated maximization problem.
	
	When $p\neq0$, consider the functional $\varphi:C^+(\sn)\rightarrow\mathbb{R}$ by
	\[
	\varphi(f):=-\frac{1}{p}\int_{\sn}f^p(v)d\mu(v).
	\]
	Their associated maximization problem is
	\begin{equation}\label{associated maximization problem}
		\max\{\varphi([f]):f\in C^+(\sn)\ \text{with}\ G_{\alpha,q}([f])=c \}
	\end{equation}
	for some positive constant $c\in(0,1)$.
	
	\begin{lem}\label{the solution of maximization problem is the solution of a normalized Minkowski problem for p<0}
		For $p\in\mathbb{R}\setminus\{0\}$, if $\mu$ is a non-zero finite Borel measure on $\sn$, and $g\in C^+(\sn)$ is a maximizer to the maximization problem to \eqref{associated maximization problem}. Then there exists a convex body $K_0\in\kno$ such that $g=h_{K_0}$ with $G_{\alpha,q}(K_0)=c$, and
		\[
		\frac{S_{p,\alpha,q}(K_0,\cdot)}{S_{p,\alpha,q}(K_0,\sn)}=\frac{\mu}{|\mu|}.
		\]
	\end{lem}
	\begin{proof}
		By the definition of Wulff shape
		\[
		[f]=\bigcap_{v\in\sn}\{x\in\rn:(x,v)\leq f(v)\},
		\]
		it is easy to see that $h_{[f]}\leq f$, and $[h_{[f]}]=[f]$.
		
		When $p\neq0$, we have
		\begin{equation}
			-\frac{1}{p}\int_{\sn}h^p_{[f]}(v)d\mu(v)\geq-\frac{1}{p}\int_{\sn}f^p(v)d\mu(v),
		\end{equation}
		which means $\varphi(h_{[f]})\geq \varphi(f)$. Since $g$ is a maximizer, $g$ must be the support function of a convex body $K_0\in\kno$, i.e., $g=h_{K_0}$, and satisfies
		\begin{equation}\label{equivalent form}
			\varphi(h_{K_0})=\max\{\varphi([f]):f\in C^+(\sn)\ \text{with} \ G_{\alpha,q}([f])=c\}.
		\end{equation}
Let $h_{t,\epsilon}=(h_{K_0}^p+tf^p+\epsilon)^{\frac{1}{p}}$. We define
$$ \Phi(t,\epsilon)=\varphi(h_{t,\epsilon})
$$
and
$$
\Psi(t,\epsilon)=G_{\alpha,q}([h_{t,\epsilon}]).
$$
It can be noticed that
$$
\Psi(0,0)=G_{\alpha,q}(K_0)=c
$$
and
$$
\Psi_\epsilon(0,0)=\frac{1}{p}S_{p,\alpha,q}(K_0,\sn)\neq 0,
$$
then when $t_0,\epsilon_0>0$ is sufficiently small, we can apply the implicit function theorem on $R=(-t_0,t_0)\times (-\epsilon_0,\epsilon_0)$, i.e., there exists $\xi\in C^1(-t_0,t_0)$ such that $(t,\xi(t))$ is the uniqueness solution of $\Psi(t,\epsilon)=c$ on $R$.
From $\Psi(t,\xi(t))=c$ , we get $\Psi_t(0,0)+\Psi_\xi(0,0)\xi'(0)=0$, i.e., $\xi'(0)=\frac{-p\Psi_t(0,0)}{S_{p,\alpha,q}(K_0,\sn)}$. Let $m(t)=\varphi(h_{t,\epsilon})$. Since $\varphi(h_{K_0})=\max\{\varphi([f]):f\in C^+(\sn)\ \text{with} \ G_{\alpha,q}([f])=c\}$, i.e., $0$ is the maximizer of $m$, we have
\begin{align*}
	0&=m'(0)\\
	&=\frac{d}{dt}|_{t=0}\Phi(t,\xi(t))\\
	&=\Phi_t(0,0)+\Phi_\xi(0,0)\xi'(0)\\
	&=\Phi_t(0,0)+\frac{|\mu|\Psi_t(0,0)}{S_{p,\alpha,q}(K_0,\sn)}\\
	&=\frac{d}{dt}|_{t=0}\varphi(h_t)+\frac{|\mu|}{S_{p,\alpha,q}(K_0,\sn)}\frac{d}{dt}|_{t=0}G_{\alpha,q}([h_t])\\
	&=\frac{-1}{p}\int_{\sn}f^p(v)d\mu(v)+\frac{|\mu|}{pS_{p,\alpha,q}(K_0,\sn)} \int_{\sn}f(v)^p dS_{p,\alpha,q}(K_0,v).\\
\end{align*}
Then by the arbitrariness of $f$, we get $\frac{\mu}{|\mu|}=\frac{S_{p,\alpha,q}(K_0,\cdot)}{S_{p,\alpha,q}(K_0,\sn)}$.
	\end{proof}
	
	We claim that Lemma \ref{the solution of maximization problem is the solution of a normalized Minkowski problem for p<0} can be extended to include the even condition if it is need, and the proof remains unchanged.
	
	\begin{lem}\label{even case for associate maximization problem}
		For $p\in\mathbb{R}\setminus\{0\}$, if $\mu$ is a non-zero finite even Borel measure on $\sn$, and $g\in C^+_e(\sn)$ is a maximizer to the maximization problem
		\[
		\max\{\varphi([f]):f\in C^+_e(\sn)\ \text{with}\ G_{\alpha,q}([f])=c \}
		\]
		for some positive constant $c\in(0,1)$. Then there exists an o-symmetric convex body $K_0\in\kne$ such that $g=h_{K_0}$ with $G_{\alpha,q}(K_0)=c$, and
		\[
		\frac{S_{p,\alpha,q}(K,\cdot)}{S_{p,\alpha,q}(K,\sn)}=\frac{\mu}{|\mu|}.
		\]
	\end{lem}

	\subsection{Existence of Maximizers to the Maximization Problems}
	We establish in this subsection the existence of maximizers for the associated maximization problems introduced in  the preceding subsection.
	\begin{lem}\label{existence of maximizer for p>0}
		For $p>0$, if $\mu$ is a non-zero finite Borel measure on $\sn$ and not concentrated on any closed hemisphere, then for every $c\in[\frac{1}{2},1)$ there exists a $K\in\kno$ with $G_{\alpha,q}(K)=c$ such that
		\[
		\varphi(h_K)=\max\{\varphi([f]):f\in C^+(\sn)\ \text{with}\ G_{\alpha,q}([f])=c \}.
		\]
	\end{lem}
	\begin{proof}
		Assume that $\{Q_l\}\subset\kno$ is a maximizing sequence of the maximization problem, that is
		\begin{equation}\label{maximizing sequence for p>0}
			\lim_{l\rightarrow\infty}\varphi(Q_l)=\max\{\varphi([f]):f\in C^+(\sn)\ \text{with}\ G_{\alpha,q}([f])=c \}>-\infty,
		\end{equation}
		where the last inequality is from the fact that the ball $r_0B$ such that
		\[
		h_{r_0B}\in\{f\in C^+(\sn)\ \text{with}\ G_{\alpha,q}([f])=c \}
		\]
		with $G_{\alpha,q}([r_0B])=c$, and $\varphi(h_{r_0B})=-\frac{1}{p}r_0^p|\mu|>-\infty$.
		
		We claim that $Q_l$ is bounded from above and below. If this do, by the upper bound of $Q_l$ and the Blaschke's selection theorem, there exists a convergent subsequence (which also denote as $Q_l$) that converges to a compact convex set $K\subset\rn$, where $G_{\alpha,q}(K)=c$ due to the continuity of $G_{\alpha,q}$. Furthermore, the positive lower bound of $Q_l$ implies that $K$ contains origin as its interior point, which means
		\[
		K\in\kno.
		\]
		Thus, we obtain the desired result.
		
		We now prove the claim by a contradiction. Assume that $h_{Q_l}$ attains  its maximum at $u_l$. If $Q_l$ is not bounded from above, we can select a subsequence $h_{Q_l}(u_l)\rightarrow\infty$ as $l\rightarrow\infty$. For every $Q_l\in\kno$, by the definition of support function,
		\begin{equation}\label{support function inequality for nonsymmetric}
			h_{Q_l}(v)\geq h_{Q_l}(u_l)(v\cdot u_l)_+.
		\end{equation}
		Moreover, since $\mu$ is not concentrated on any closed hemisphere, the following
		\begin{equation}\label{concentrate inequality for nonsymmetric}
			\int_{\sn}(u_l\cdot v)_+^pd\mu(v)\geq c_0>0
		\end{equation}
		holds for some positive constant $c_0$. Combining \eqref{support function inequality for nonsymmetric} and \eqref{concentrate inequality for nonsymmetric}, we conclude that
		\begin{align*}
			\varphi(Q_l)&=-\frac{1}{p}\int_{\sn}h_{Q_l}(v)^pd\mu(v)\\
			&\leq-\frac{1}{p}\int_{\sn}h_{Q_l}(u_l)^p(u_l\cdot v)_+^pd\mu(v)\\
			&\leq-\frac{1}{p}h_{Q_l}(u_l)^pc_0\\
			&\rightarrow-\infty
		\end{align*}
		as $l\rightarrow\infty$, but this contradicts equation \eqref{maximizing sequence for p>0}. Therefore, we conclude that $Q_l$ is uniformly bounded from above.
		
		Moreover, if $Q_l$ does not have a uniformly positive lower bound, by selecting a subsequence of $Q_l$,   where $Q_l\rightarrow K$ with $G_{\alpha,q}(Q_l)=c$, there exists a $v_l$ such that $h_{Q_l}(v_l)\rightarrow0$ as $l\rightarrow\infty$. For any $\epsilon>0$, $Q_l\subset\{x|x\cdot v_l>-\epsilon\}:=H_\epsilon$. Since $K$ is bounded from above, there exists a large radius $R>0$ such that $K\subset B_R\cap H_\epsilon$. Denote $H:=\lim_{\epsilon\rightarrow0}H_\epsilon$. Thus
		\begin{align*}
			c&\geq\frac{1}{2}=\int_{H}g_{\alpha,q}(x)dx\\
			&=\int_{H\setminus B_R}g_{\alpha,q}(x)dx+\int_{H\cap B_R }g_{\alpha,q}(x)dx\\
			&>\int_{H\cap B_R }g_{\alpha,q}(x)dx\\
			&\geq G_{\alpha,q}(Q_l),
		\end{align*}
		but this contradicts $G_{\alpha,q}(Q_l)=c$. Therefore, $Q_l$ is bounded from below.
	\end{proof}
	
	\begin{lem}\label{existence of the maximizerfor p<=0}
		For $p<0$, $\alpha>0$ and $q<\frac{\alpha}{n+\alpha}$, if $\mu$ is a non-zero finite even Borel measure on $\sn$ and vanishes on all great subspheres, then for every $c\in(0,1)$ there exists a $K\in\kne$ with $G_{\alpha,q}(K)=c$ such that
		\[
		\varphi(h_K)=\max\{\varphi([f]):f\in C^+_e(\sn)\ \text{with}\ G_{\alpha,q}([f])=c \}.
		\]
	\end{lem}
	
	\begin{proof}
		Assume that $\{Q_l\}\subset\kne$ is a maximizing sequence of the maximization problem, that is,
		\begin{equation}\label{maximizing sequence}
			\lim_{l\rightarrow\infty}\varphi(Q_l)=\max\{\varphi([f]):f\in C^+_e(\sn)\ \text{with}\ G_{\alpha,q}([f])=c \}>0,
		\end{equation}   	
		where the last inequality arises for the same reason as \eqref{maximizing sequence for p>0}. We claim that $Q_l$ is bounded from above and below. If the claim holds ture, the argument parallels Lemma \ref{existence of maximizer for p>0}, indicating that
		\[
		K\in\kne,
		\]
		thus we obtain the result.
		
		We now prove the claim. Firstly, we prove the positive lower bound. Otherwise, we can select a subsequence of $\{Q_l\}$, also denote by $Q_l$, satisfies
		\[
		h_{Q_l}(v_l)\rightarrow0
		\]
		as $l\rightarrow\infty$, where $h_{Q_l}(v_l)=\min_{v\in\sn}h_{Q_l}(v)$. Since $Q_l$ is o-symmetric,
		\[
		Q_l\subset\{x\in\rn:|(x,v_l)|\leq h_{Q_l}(v_l)\}:=W_l.
		\]
		For every $l$, by selecting appropriate coordinate systems we have
		\begin{align*}
			G_{\alpha,q}(Q_l)&\leq G_{\alpha,q}(W_l)\\ &=\frac{1}{Z(\alpha,q)}\int_{W_l}[1-\frac{q}{\alpha}(\sum_{i=1}^{n}x_i^2)^{\frac{\alpha}{2}}]_+^{\frac{1}{q}-\frac{n}{\alpha}-1}dx\\
			&\leq \frac{1}{Z(\alpha,q)}\int_{W_l}[1-\frac{q}{\alpha}(\sum_{i=1}^{n-1}x_i^2)^{\frac{\alpha}{2}}]_+^{\frac{1}{q}-\frac{n}{\alpha}-1}dx\\
			&=\frac{2h_{Q_l}(v_l)}{Z(\alpha,q)}\int_{W_l\cap\mathbb{R}^{n-1}}[1-\frac{q}{\alpha}(\sum_{i=1}^{n-1}x_i^2)^{\frac{\alpha}{2}}]_+^{\frac{1}{q}-\frac{n}{\alpha}-1}dx_1\cdots dx_{n-1}\\
			&=\frac{2Ch_{Q_l}(v_l)}{Z(\alpha,q)}\rightarrow0
		\end{align*}
		as $l\rightarrow\infty$, where $C$ is a finite positive constant, and the second inequality is because $q<\frac{\alpha}{n+\alpha}$. This contradicts $G_{\alpha,q}(Q_l)=c>0$ and we derive the positive lower bound.
		
		Next we prove the upper bound. For $v_0\in\sn$ and $r\in(0,1)$, define
		\[
		\omega_r(v_0)=\{u\in\sn:|(u,v_0)|> r\}.
		\]
		If $Q_l$ is not bounded from above, there exists a subsequence of $\{Q_l\}$, also denote by $Q_l$, its polar body such that $Q^*_l$ converges to a convex body $Q^*_0$, and there exist some $v_0$ satisfy $h_{Q^*_0}(v_0)=0$. By Lemma 6.2 in \cite{HLYZ18}, we have
		\begin{equation}\label{Lp alexsandrov}
			\rho_{Q^*_l}\rightarrow0
		\end{equation}
		on $\omega_r(v_0)$. Since $h_{Q_l}$ has the lower bound, $h_{Q^*_l}$ has the upper bound, i.e., there exists a $R>0$ such that $Q^*_l\subset RB$, where $B$ is the unit ball in $\rn$. Then
		\begin{align}\label{star body estimate}
			0\leq\int_{\sn}h^p_{Q_l}(v)d\mu(v)&=\int_{\omega_r(v_0)}\rho^{-p}_{Q^*_l}(v)d\mu(v)+\int_{\sn\setminus\omega_r(v_0)}\rho^{-p}_{Q^*_l}(v)d\mu(v)\notag\\
			&\leq\int_{\omega_r(v_0)}\rho^{-p}_{Q^*_l}(v)d\mu(v)+R^{-p}\mu(\sn\setminus\omega_r(v_0)).
		\end{align}
		Note that $\mu$ vanishes on all great subspheres, thus
		\begin{equation}\label{vanishes on all gret subspheres}
			\lim_{r\rightarrow0}\mu(\sn\setminus\omega_r(v_0))=\mu(\sn\cap v_0^\bot)=0.
		\end{equation}
		Combining \eqref{Lp alexsandrov}, \eqref{star body estimate} and \eqref{vanishes on all gret subspheres}, we have
		\[
		\lim_{l\rightarrow\infty}\int_{\sn}h^p_{Q_l}(v)d\mu(v)=0,
		\]
		which contradicts \eqref{maximizing sequence}. Thus we derive the upper bound.
	\end{proof}
	
	\subsection{Existence of Weak Solutions to the Normalized Minkowski Problem}
	
	Now we are prepared to prove the existence of solutions to the Minkowski problem, as stated in  Theorem \ref{p>0 existence of N-Minkowski problem} and Theorem \ref{p<=0 existence of N-Minkowski problem}.
	\begin{proof}[Proof of Theorem \ref{p>0 existence of N-Minkowski problem}]
		The sufficiency follows from Lemma \ref{the solution of maximization problem is the solution of a normalized Minkowski problem for p<0} and Lemma \ref{existence of maximizer for p>0}. For the necessity, if there exists a pair $(K,\lambda)$ such that $K\in\kno$, $\lambda>0$ and $S_{p,\alpha,q}(K,\cdot)=\lambda\mu$, then by the definition of $S_{p,\alpha,q}$, we have
		\begin{equation*}
			\int_{\sn}f(v)dS_{p,\alpha,q}(K,v)=\int_{\partial K}f(\nu_K(x))(x\cdot \nu_K(x) )^{1-p}g_{\alpha,q}(x)d\hm(x).
		\end{equation*}
		for any Borel function $f:\sn\rightarrow R$. Let $M_K(x)=(x\cdot \nu_K(x) )^{1-p}g_{\alpha,q}(x)$ and $m=\min_{x\in \partial K}M_K(x)$. Since $K\in\kno,$ we have $m>0.$ Thus
		\begin{equation*}
			\begin{split}
				\int_{\sn}(u\cdot v)_+dS_{p,\alpha,q}(K,v)&=\int_{\partial K}(u\cdot\nu_K(x))_+M_K(x)d\hm(x)\\
				&\geq m\int_{\partial K}(u\cdot\nu_K(x))_+d\hm(x)\\
				&= m\int_{\sn}(u\cdot v)_+dS(K,v)>0,
			\end{split}
		\end{equation*}
		where the final inequality arises from the condition that $S(K,v)$ is not concentrated on any closed hemisphere, implying that $\mu$ is not concentrated on any closed hemisphere.		
	\end{proof}

	\begin{proof}[Proof of Theorem \ref{p<=0 existence of N-Minkowski problem}]
		We prove the sufficiency by an approximation argument. Suppose that there eixsts a non-zero finite absolutely continuous measure sequence $\mu_i\rightarrow\mu$ weakly, i.e., $d\mu_i=f_id\nu$, and $f_i$ is smooth even functions on $\sn$. Given the properties of the Hausdorff measure, $\mu_i$ vanishes on all great subsphere. By combining Lemma \ref{even case for associate maximization problem} and Lemma \ref{existence of the maximizerfor p<=0}, we conclude that there exists a convex body $K_i\in\kne$ such that
		\[
		S_{p,\alpha,q}(K_i,\cdot)=\lambda_i\mu_i,
		\]
		which means
		\begin{equation}\label{monge ampere equation}
			\frac{1}{Z(\alpha,q)}\int_{\sn}h^{1-p}_{K_i}[1-\frac{q}{\alpha}(h^2_{K_i}+|\nabla h_{K_i}|^2)^{\frac{\alpha}{2}}]_+^{\frac{1}{q}-\frac{n}{\alpha}-1}\det(\nabla^2h_{K_i}+h_{K_i}I)d\nu=\lambda_i\int_{\sn}f_id\nu.
		\end{equation}
		
		If $f_i\in C^\alpha(\sn)$, the solution $h_{K_i}\in C^{2,\alpha}(\sn)$, and see the regularity result in Lemma \ref{$C^2$ estimate}.  We claim that by selecting a subsequence, $h_{K_i}$ is bounded from above and below. If the claim is established, then by the upper bound, we can apply the Blaschke's selection theorem, which guarantees the existence of a subsequence of $K_i$ converging to a centrally o-symmetric convex body $K$, and the lower bound implies that $K\in\kne$. Then by the weak convergence of $S_{p,\alpha,q}(K_i,\cdot)$ and $\mu_i$, we have $S_{p,\alpha,q}=\lambda\mu$.

		We prove the claim. The lower bound follows the same reasoning as Lemma \ref{existence of the maximizerfor p<=0}, and the upper bound is established through contradiction.
		Suppose that $h_{K_i}$ attains its maximum on $\sn$ at $u_i$. Since $\mu$ is not concentrated on any great subsphere, for every $u_i$ there exist $\epsilon,\delta>0$ such that
		\begin{equation}\label{mu is not concentrated on any subsphere}
			\int_{\{v\in\sn:|v\cdot u_i|>\delta\}}f_i(v)d\nu(v)>\epsilon>0
		\end{equation}
		for large enough $i$. Moreover, by the definition of support function,
		\begin{equation}\label{support function inequality2}
			h_{K_i}(u)\geq h_{K_i}(u_i)|(u,u_i)|.
		\end{equation}
		Combining \eqref{monge ampere equation}, \eqref{mu is not concentrated on any subsphere}, \eqref{support function inequality2} and the monotonicity of $t$ of the function $[1-\frac{q}{\alpha} t^\alpha]^{\frac{1}{q}-\frac{n}{\alpha}-1}_+$ when $q<\frac{\alpha}{n+\alpha}$,
		\begin{align*}
			0&<\epsilon\\
			&<\int_{\{v\in\sn:|v\cdot u_i|>\delta\}}f_i(v)dv\\
			&=\frac{1}{\lambda_i Z(\alpha,q)}\int_{\{v\in\sn:|v\cdot u_i|>\delta\}}h^{1-p}_{K_i}[1-\frac{q}{\alpha}(h^2_{K_i}+|\nabla h_{K_i}|^2)^{\frac{\alpha}{2}}]_+^{\frac{1}{q}-\frac{n}{\alpha}-1}\det(\nabla^2h_{K_i}+h_{K_i}I)dv\\
			&\leq\frac{1}{\lambda_i Z(\alpha,q)}h^{1-p}_{K_i}(u_i)[1-\frac{q}{\alpha}h^\alpha_{K_i}(u_i)\delta^\alpha]_+^{\frac{1}{q}-\frac{n}{\alpha}-1}\int_{\{v\in\sn:|v\cdot u_i|>\delta\}}\frac{(h^2_{K_i}(v)+|\nabla h_{K_i}|^2)^{\frac{n}{2}}}{h_{K_i}(v)}dv\\
			&\leq\frac{1}{\lambda_i Z(\alpha,q)}h^{-p}_{K_i}(u_i)[1-\frac{q}{\alpha}h^\alpha_{K_i}(u_i)\delta^\alpha]_+^{\frac{1}{q}-\frac{n}{\alpha}-1}\delta^{-1}\int_{\sn}(h^2_{K_i}(v)+|\nabla h_{K_i}|^2)^{\frac{n}{2}}dv\\
			&\leq\frac{1}{\lambda_i Z(\alpha,q)}h^{n-p}_{K_i}(u_i)[1-\frac{q}{\alpha}h^\alpha_{K_i}(u_i)\delta^\alpha]_+^{\frac{1}{q}-\frac{n}{\alpha}-1}\delta^{-1}\mathcal{H}^{n-1}(\sn)\rightarrow0
		\end{align*}
		as $h_{K_i}\rightarrow\infty$, where the last limit provided by $q<0$ with $\frac{\alpha}{q}-\alpha<p<0$ or $0\leq q<\frac{\alpha}{n+\alpha}$ with $p<0$, which is a contradiction. Thus we prove the sufficiency.
		
		For the necessity, the proof follows a similar approach to Theorem  \ref{p>0 existence of N-Minkowski problem}.
	\end{proof}
	


	

	\section{Brunn-Minkowski inequalities and isoperimetric inequalities}\label{B-M section}
	\subsection{Brunn-Minkowski Inequality}
	In \cite{C23}, Cordero-Erausquin and Rotem proved that the Gardner-Zvavitch
conjecture \cite{GZ10} is true for all rotationally invariant log-concave measures. This can be
extended to a broader class of densities which includes some cases of the generalized
Gaussian density.
	\begin{thm}\cite{C23}\label{Brunn minkowski inequality} Let $\omega:[0,\infty)\rightarrow (-\infty,\infty]$ be a non-decreasing function such that $t\rightarrow \omega(e^t)$ is convex and let $\mu$ be the measure on $\rn$ with density $e^{-\omega(|x|)}.$ Then for every symmetric convex bodies $K,L\subset \rn$ and $\lambda\in[0,1]$,
		$$
		\mu((1-\lambda)K+\lambda L)^{\frac{1}{n}}\geq (1-\lambda)\mu(K)^{\frac{1}{n}}+\lambda\mu(L)^{\frac{1}{n}}.
		$$
	\end{thm}

	For generalized Gaussian density $g_{\alpha,q}$, we verify that $g_{\alpha,q}$ satisfies the condition in Theorem \ref{Brunn minkowski inequality}. Thus the generalized Gaussian measure $G_{\alpha,q}$ satisfies the Brunn-Minkowski inequality in the o-symmetric setting.
	
	\begin{thm}[\bf Brunn-Minkowski inequality]\label{BMI}
		Suppose $\alpha>0$ and $q<\frac{\alpha}{n+\alpha}$. Then for o-symmetric convex bodies $K,L\in \kne$ and $\lambda\in[0,1]$,
		\begin{equation}\label{uni}
			G_{\alpha,q}((1-\lambda)K+\lambda L)^{\frac{1}{n}}\geq (1-\lambda)G_{\alpha,q}(K)^{\frac{1}{n}}+\lambda G_{\alpha,q}(L)^{\frac{1}{n}}.
		\end{equation}
	\end{thm}
	\begin{proof}
		 By the definition of generalized Gaussian density $g_{\alpha,q}$,  set $\omega:[0,\infty)\rightarrow(-\infty,\infty]$ by
		$$
		\omega(t)=(\frac{n}{\alpha}+1-\frac{1}{q})\log[1-\frac{q}{\alpha}t^{\alpha}]_++C,
		$$
		where $C=-\log Z(\alpha,q)$. Trivially, $\omega$ is a non-decreasing function provided $q<\frac{\alpha}{n+\alpha}$. Thus
		$$\omega(e^t)=(\frac{n}{\alpha}+1-\frac{1}{q})\log[1-\frac{q}{\alpha}e^{t\alpha}]_++C.$$
		By direct calculation, it can be concluded that
		\begin{equation*}
			\omega'(e^t)=\begin{cases}
				0 ,&\text{if}\ q>0\ \text{with}\  t\geq\frac{1}{\alpha}\log\frac{\alpha}{q} \\
				\frac{(\alpha-\alpha q-nq)e^{t\alpha}}{\alpha-qe^{t\alpha}},\quad &\text{otherwise}
			\end{cases}	
		\end{equation*}
		and
		\begin{equation*}
			\omega''(e^t)=\begin{cases}
				0,\quad &\text{if}\ q>0\ \text{with}\  t\geq\frac{1}{\alpha}\log\frac{\alpha}{q} \\
				\frac{\alpha^2(\alpha-\alpha q-nq)e^{t\alpha}}{[\alpha-qe^{t\alpha}]^2},\quad &\text{otherwise}
			\end{cases}	
		\end{equation*}
		then $\omega''(e^t)\geq 0$, i.e., $\omega(e^t)$ is convex with respect to $t$. Then by Theorem \ref{Brunn minkowski inequality}, Brunn-Minkowski inequality for generalized Gaussian volume is established.
	\end{proof}
	\subsection{Isoperimetric Inequality}\label{isoperimetric subsection}
	
	For a given Borel probability measure $\mu$ on $\rn$, its surface measure $\mu^+$ is defined on Borel set $A\subset\rn$,  $\mu^+[A]:=\liminf_{\epsilon \rightarrow 0}\frac{\mu[A_\epsilon]-\mu[A]}{\epsilon}$, where $A_\epsilon=A+B_\epsilon$ and $B_\epsilon=\epsilon B$ for the unit ball $B$. The isoperimetric profile $I[\mu]$ of $\mu$ is a function on
	$[0,1]$ defined by
	$$
	I[\mu](a):=\inf\{\mu^+[A]|A\subset\rn\ \text{with}\  \mu[A]=a\}.
	$$
	
	We are in a
position to present three classes of isoperimetric inequalities applicable to the generalized Gaussian density. These inequalities are determined by the log-concavity and
log-convexity of the density function $g_{\alpha,q}$  for different parameters $\alpha,q$ as follows.\\
	(1) $\log$-concave: $q\in[0,\frac{\alpha}{n+\alpha})$ with $\alpha\geq1$;\\
	(2) $\log$-convex: $q\leq0$ with $\alpha\in(0,1]$ or $q\in[\frac{\alpha}{n+\alpha},\frac{\alpha}{n})$ with $\alpha\geq1$;\\
	(3) Poincar\'e limit: $q\in(0,\frac{\alpha}{n+\alpha})$ with $\alpha>0$.\\
	It is necessary to point out that the isoperimetric inequality for indices within the following range has not been provided: $q<0$ with $\alpha>1$ and $q\in[\frac{\alpha}{n+\alpha},\frac{\alpha}{n})$ with $\alpha\in(0,1)$.
	
	Nolwen \cite{N} provided the isoperimetric profile for $\log$-concave density.
	\begin{thm}\cite{N}\label{log concave} There exists a universal constant $c>0$ such that,  all
		$\log$-concave measures $\mu$ on $\rn$
		with spherically symmetric
		and isotropic satisfy the following
		isoperimetric
		inequality:
		$$
		\mu^+(A)\geq c\min(\mu(A),1-\mu(A)).
		$$
		
	\end{thm}
	
	For $\log$-convex density, Gregory \cite{G16} provided the proof of $\log$-convex density conjecture, which means that the balls centered on the origin are isoperimetric regions.
	\begin{thm}\cite{G16}\label{log-convex} Given a density $f(x)=e^{g(|x|)}$ on $\rn$
		with $g$ smooth, convex and even, balls around the origin are isoperimetric regions with
		respect to weighted perimeter and volume.
	\end{thm}
	
	When the density is neither $\log$-convex nor $\log$-concave, we consider utilizing a generalized Poincar\'e limit, which was first proposed by Takatsu \cite{T14}. It is asserted that the generalized Gaussian measure can be approximated by the projections of the uniform radial probability measure on $\sn$. To demonstrate the isoperimetric profile in this context, we still require some definitions and lemmas from \cite{T14}.
	
	Define the radial probability measure  $\mu_n^f$ with density $f$ as absolutely continuous probability measure on $\rn$ with density
	$$
	\frac{d \mu_n^f}{dx}(x)=\frac{1}{M_n^f}f(|x|),
	$$
	where $M_n^f=n\omega_n\int_0^\infty f(r)r^{n-1}dr$. Denote by $r_f:=\inf\{r|r\in   \text{supp}(f)\}$, $R_f:=\sup\{r|r\in \text{supp}(f)\}$.
	
	Call a measure $\rho$ satisfies $C_n$, if it satisfies: $s^\rho_n$ is Lipschitz continuous, where $s^\rho_n(x)=\rho(|x|)x$ when $x\neq0$, and equal to $0$ when $x=0$.
	
	\begin{thm}\cite{T14}\label{c}
		Let $\mu_n^f$
		be the generalized Poincar\'e limit with $\rho$ satisfying $C_n$ and denote
		by L the smallest Lipschitz constant of $S_n^\rho$. We then have $I[\mu_n^f](a)\geq \frac{I[\gamma_1](a)}{L}$  for any $a\in[0,1]$ and $L\geq (\frac{\sqrt{2\pi}^n}{M_n^f}\liminf_{r\rightarrow 0}f(r))^{-\frac{1}{n}}.$
		
	\end{thm}
	\begin{prop}\cite{T14}\label{d} A radial probability measure $\mu_n^f$ is a generalized Poincar\'e limit if and
		only if supp$(f)$ is connected, on the interior of which $f$ is continuous.
	\end{prop}
	\begin{prop}\cite{T14}\label{g} For the generalized Poincar\'e limit $\mu_n^f$, assume $\limsup_{r\rightarrow 0}f(r)< \infty$. Then $C_n$ condition
		is equivalent to the combination of $(a)$ and $(b)$ as follows:

		$(a)$ $f$ is positive on $(0,R_f)$ and $\liminf_{r\rightarrow 0}f(r)>0$.

		$(b)$ There exists a continuous, positive function $\psi$ on $(R,R_f)$ for some $R> 0$ such that

		$(b_1)$ $\lambda f(r)\psi(r)r^{n-1}\leq \int_r^{R_f}f(s)s^{n-1}ds\leq\frac{1}{\lambda}f(r)\psi(r)r^{n-1}$ for some $\lambda\in (0,1)$.

		$(b_2)$ $\liminf_{r\rightarrow R_f}\{\psi(r)^2\ln(f(r)\psi(r)r^{n-1})\}> -\infty.$
	\end{prop}
	\begin{lem}\cite{T14}\label{f} For a radial probability measure $\mu_n^f$, suppose that $f$ is $C^2$, positive on $(R,R_f)$ for some $R> 0$ and $\lim_{r\rightarrow R_f}f(r)=0$. Moreover if the function $\Phi=- \log f$ satisfies $\lim_{r\rightarrow R_f}\Phi''(r)\in(0,\infty]$, $\lim_{r\rightarrow R_f}\Phi'(r)=\infty$ and $\limsup_{r\rightarrow R_f}\frac{\Phi''(r)}{\Phi'(r)^2}< \infty$, then $(b)$ holds.
	\end{lem}

	Applying the lemmas mentioned above, we can derive the isoperimetric profile for the generalized Gaussian measure $G_{\alpha,q}$ when $q\in(0,\frac{\alpha}{n+\alpha})$ with $\alpha>0$.
	\begin{thm}\label{neither log-convex nor log-concave}For $q\in(0,\frac{\alpha}{n+\alpha})$ with $\alpha>0$, the generalized Gaussian measure $G_{\alpha,q}$ satisfies the conditions in Theorem \ref{c}. Then for any convex body $A$ with $G_{\alpha,q}(A)=\frac{1}{2}$, we have
		\begin{equation}\label{e}
			G_{\alpha,q}^+(A)\geq \frac{I[\gamma_1](\frac{1}{2})}{L}=\frac{1}{\sqrt{2\pi}L}.
		\end{equation}

	\end{thm}
	\begin{proof}

		Since $q>0,$ we can derive that $r_{g_{\alpha,q}}=0$ and $R_{g_{\alpha,q}}=(\frac{\alpha}{q})^{\frac{1}{\alpha}}$ through a simple calculation.
		
		Using Proposition \ref{d}, it can be verified that the radial probability measure $G_{\alpha,q}$ is  generalized Poincar\'e limit. Since $0<\lim_{r\rightarrow 0}g_{\alpha,q}(r)< \infty$ and $g_{\alpha,q}$ is $C^2$, positive on $(R,R_{g_{\alpha,q}})$ for some $R>0$, $(a)$ holds.
		
		For $(b)$, $\lim_{r\rightarrow (\frac{\alpha}{q})^{\frac{1}{\alpha}}}g_{\alpha,q}(r)=0$. Moreover, set $\Phi(r)\equiv-\log [1-\frac{q}{\alpha} r^\alpha]_+^{\frac{1}{q}-\frac{n}{\alpha}-1}=(\frac{n}{\alpha}+1-\frac{1}{q})\log[1-\frac{q}{\alpha} r^\alpha]_+$. Then
		$$
		\Phi'(r)=(\frac{1}{q}-\frac{n}{\alpha}-1)\frac{qr^{\alpha-1}}{1-\frac{q}{\alpha}r^\alpha},
		$$
		and
		$$
		\Phi''(r)=(\frac{1}{q}-\frac{n}{\alpha}-1)\frac{qr^{\alpha-2}(\alpha-1+\frac{q}{\alpha}r^\alpha)}{(1-\frac{q}{\alpha}r^\alpha)^2}.
		$$
		Thus $\lim_{r\rightarrow R_{g_{\alpha,q}}}\Phi'(r)=\infty$, $\lim_{r\rightarrow R_{g_{\alpha,q}}}\Phi''(r)=\infty\in(0,\infty]$, and $\limsup_{r\rightarrow R_{g_{\alpha,q}}}\frac{\Phi''(r)}{\Phi'(r)^2}=\frac{1}{\frac{1}{q}-\frac{n}{\alpha}-1}<\infty$. Then by using Lemma \ref{f}, $(b)$ holds.
		Combining the Proposition \ref{g} and Theorem \ref{c}, inequality \eqref{e} is obtained.
	\end{proof}
	For $L_p$ Gaussian surface area measure $S_{p,\alpha,q}$, we derive the isoperimetric inequality for $p\geq1$ by the H\"{o}lder inequality, which is motivated by \cite{f23}.
	\begin{lem}\label{lp generalized gaussian isoperimetric inequality}$($\textbf{$L_p$ generalized Gaussian isoperimetric inequality}$)$ For $p\geq1$ and $q<\frac{\alpha}{n+\alpha}$, let $K\in\kno$. Then
		$$
		|S_{p,\alpha,q}(K)|\geq(nG_{\alpha,q}(K))^{1-p}|S_{\alpha,q}(K)|^p.
		$$
	\end{lem}
	\begin{proof}
		Since
		\begin{equation}\label{2}
			\begin{split}
				&\frac{1}{Z(\alpha,q)}\int_{\partial K}(x\cdot v)[1-\frac{q}{\alpha}|x|^\alpha]_+^{\frac{1}{q}-\frac{n}{\alpha}-1}d\hm(x)\\
				&=\frac{1}{Z(\alpha,q)}\int_{K}div(x[1-\frac{q}{\alpha}|x|^\alpha]_+^{\frac{1}{q}-\frac{n}{\alpha}-1})dx\\
				&=nG_{\alpha,q}(K)-\frac{1-\frac{qn}{\alpha}-q}{Z(\alpha,q)}\int_K [1-\frac{q}{\alpha}|x|^\alpha]_+^{\frac{1}{q}-\frac{n}{\alpha}-2}|x|^\alpha dx,
			\end{split}
		\end{equation}
		we set $\widetilde{G_{\alpha,q}}(K)=\frac{1}{nZ(\alpha,q)}\int_{\partial K}(x\cdot v)[1-\frac{q}{\alpha}|x|^\alpha]_+^{\frac{1}{q}-\frac{n}{\alpha}-1}d\hm(x)$. By \eqref{2} and $q<\frac{\alpha}{n+\alpha}$, we have
		$$
		G_{\alpha,q}(K)\geq \widetilde{G_{\alpha,q}}(K).
		$$
		Define $d{G_{\alpha,q}^*}=\frac{1}{nZ(\alpha,q)\widetilde{G_{\alpha,q}}(K)}(x\cdot v)[1-\frac{q}{\alpha}|x|^\alpha]_+^{\frac{1}{q}-\frac{n}{\alpha}-1}d\hm(x)$,
		which is a probability measure. For $p\geq 1,$ using  the H\"older inequality
		\begin{equation*}
			\begin{split}
				\left(\frac{|S_{p,\alpha,q}(K)|}{n\widetilde{G_{\alpha,q}}(K)}\right)^{-\frac{1}{p}}&=\left(\frac{1}{nZ(\alpha,q)\widetilde{G_{\alpha,q}}(K)}\int_{\partial K}(x\cdot v)^{1-p}[1-\frac{q}{\alpha}|x|^\alpha]_+^{\frac{1}{q}-\frac{n}{\alpha}-1}d\hm(x)\right)^{-\frac{1}{p}}\\
				&=\left(\int_{\partial K}(x\cdot v)^{-p}dG_{\alpha,q}^*\right)^{-\frac{1}{p}}\\
				&\leq\left(\int_{\partial K}(x\cdot v)^{-1}dG_{\alpha,q}^*\right)^{-1}\\
				&=\left(\frac{1}{nZ(\alpha,q)\widetilde{G_{\alpha,q}}(K)}\int_{\partial K}[1-\frac{q}{\alpha}|x|^\alpha]_+^{\frac{1}{q}-\frac{n}{\alpha}-1}d\hm(x)\right)^{-1}\\
				&=\left(\frac{|S_{\alpha,q}(K)|}{n\widetilde{G_{\alpha,q}}(K)}\right)^{-1},
			\end{split}
		\end{equation*}
		then
		$$
		|S_{p,\alpha,q}(K)|\geq(nG_{\alpha,q}(K))^{1-p}|S_{\alpha,q}(K)|^{p}.
		$$
	\end{proof}
	
	When generalized Gaussian volume $G_{\alpha,q}=\frac{1}{2}$, a direct corollary follows from Theorem \ref{log concave}, Theorem \ref{log-convex}, Theorem \ref{neither log-convex nor log-concave} and Lemma \ref{lp generalized gaussian isoperimetric inequality}.
	\begin{cor}\label{isoperimetric inequalities for index}
		For $p\geq1$, suppose that $0\leq q<\frac{\alpha}{n+\alpha}$ with $\alpha>0$ or $q<0$ with $0<\alpha\leq1$. If $K\in\kno$ with $G_{\alpha,q}(K)=\frac{1}{2}$, then
		\[
		|S_{p,\alpha,q}(K)|\geq (\frac{n}{2})^{1-p}I^p[G_{\alpha,q}](\frac{1}{2})>0,
		\]
		where $I[G_{\alpha,q}](\frac{1}{2})$ is the isoperimetric profile.
	\end{cor}
	
	\section{Existence of Smooth Solutions for $p\geq n$}\label{continuity method for $p>n$}
	Motivated by \cite{f23}(in their arxiv version), we employ the continuity method to establish the existence of solutions to the associated Monge-Amp\`ere equation of $L_p$ generalized Gaussian Minkowski problem for $p\geq n$. Additionally, we utilize the maximum principle to prove uniqueness of solutions.
	
	We repeat the associated Monge-Amp\`ere equation
	\begin{equation}\label{equation q}
		h^{1-p}g_{\alpha,q}(|Dh|)\det(h_{ij}+h\delta_{ij})=f,
	\end{equation}
	and the existence and uniqueness of solutions to this equation are as follows.
	
	\begin{thm}\label{smooth solution for p>n1}
		For $\beta\in(0,1)$, $p\geq n$, $\alpha>0$ and $q<\frac{\alpha}{n+\alpha}$, suppose $f\in C_+^\beta(\sn)$ such that $\frac{1}{C}<f<C$ for some positive constant $C$ when $p>n$ and $f<\frac{1}{Z(\alpha,q)}$ when $p=n$. Then there exists a uniqueness solution $h\in C^{2,\beta}(\sn)$ to equation \eqref{equation q}.
	\end{thm}
	
	To prove Theorem \ref{smooth solution for p>n1}, we need some a priori estimates.
	
	\begin{lem}\label{C0 estimate}(\textbf{$C^0$ estimate})\label{C0 estimate for p>n}
		Suppose that $p\geq n$, $\alpha>0$ and $q<\frac{\alpha}{n+\alpha}$, and $h\in C^2_+(S^{n-1})$ is a solution to the equation \eqref{equation q}. If $f$ is a positive function on $S^{n-1}$ such that $\frac{1}{C}<f<C$	
		for some positive constant $C$ when $p>n$ and $f<\frac{1}{Z(\alpha,q)}$ when $p=n$, then there exists a positive constant $C'$(which only depends on $C$) such that
		\[
		\frac{1}{C'}<h<C'.
		\]
		Especially, for $q\in(0,\frac{\alpha}{n+\alpha})$, we have an exact upper bound
		\[
		h<(\frac{\alpha}{q})^{\frac{1}{\alpha}}.
		\]
	\end{lem}

	\begin{proof}
		We first show that $h$ is bounded from above. Assume that $h$ achieves its maximum at $u_0$ on $S^{n-1}$, then we can get that $\nabla h|_{u_0}=0$ and $\nabla^2h|_{u_0}\leq 0$. By equation \eqref{equation q}, we have
		\begin{align}\label{case1}
			f(u_0)&=\frac{1}{Z(\alpha,q)}h^{1-p}(u_0)g_{\alpha,q}(|h(u_0)|)\det(h_{ij}+h\delta_{ij})|_{u_0}\notag\\
			&\leq\frac{1}{Z(\alpha,q)}h^{n-p}(u_0)g_{\alpha,q}(|h(u_0)|).
		\end{align}
		We can get the upper bound by a contradiction. When $q\leq0$, if $h(u_0)\rightarrow\infty$ in \eqref{case1}, thus $f(u_0)\leq0$, which contradicts the condition $f>\frac{1}{C}>0$; When $0<q<\frac{\alpha}{n+\alpha}$,
		if $h(u_0)\geq(\frac{\alpha}{q})^{\frac{1}{\alpha}}$, then $g_{\alpha,q}=0$. Also in \eqref{case1}, $f(u_0)\leq0$, which is a  contradiction. Thus we get the upper
bound for solutions to \eqref{equation q} and the explicit upper bound for the case of $q>0.$
		
		Next we also show that $h$ is bounded from below. Assume that $h$ achieves its minimum at $\nu_0$, then $\nabla h|_{\nu_0}=0$, and $\nabla^2h|_{\nu_0}\geq0$. Then by equation \eqref{equation q}, we have
		\begin{equation}\label{inequality 2}
			f(\nu_0)\geq\frac{1}{Z(\alpha,q)}h^{n-p}(v_0)g_{\alpha,q}(|h(v_0)|).
		\end{equation}
		If $h(\nu_0)\rightarrow0$, it will contradict \eqref{inequality 2} and the conditions $f<C$ for $p>n$ and $f<\frac{1}{Z(\alpha,q)}$ for $p=n$.
	\end{proof}
	\begin{rem}
		It is necessary to illustrate that we can next replace $[1-\frac{q}{\alpha}h^\alpha(u_0)]_+$ by $[1-\frac{q}{\alpha}h^\alpha(u_0)]$ because of the exact upper bound for $q\in(0,\frac{\alpha}{n+\alpha})$.
	\end{rem}
	\begin{lem}\label{$C^2$ estimate}\textbf{($C^2$ estimate)}
		For $\beta\in(0,1)$, $p\geq n$, $\alpha>0$ and $q<\frac{\alpha}{n+\alpha}$, suppose $f\in C^\beta(S^{n-1})$ such that  $\frac{1}{C}<f<C$ for some positive constants $C$ when $p>n$ and $f<\frac{1}{Z(\alpha,q)}$ when $p=n$. If $h\in C^{2,\beta}(S^{n-1})$ is a solution to the equation \eqref{equation q}, then there exists some positive constant $C'$ which only depends on $C$ such that
		\[
		|h|_{C^{2,\beta}}\leq C'.
		\]
	\end{lem}
	
	\begin{proof}
	 Since Lemma \ref{C0 estimate for p>n}, there exists $\tau_0$ such that
		\begin{equation}\label{C0 estimate in C2 estimate}
			\frac{1}{\tau_0}<h<\tau_0.
		\end{equation}
		Consider the function $h^2+|\nabla h|^2$, assume its maximum attains at $v_0$. Thus we have
		\[
		(h_{ij}+h\delta_{ij})h_j=0.
		\]
		we discuss the properties of the above equation through a small perturbation $\epsilon>0$. If $\det(h_{ij}+h\delta_{ij})\neq0$, that means $h_j(u_0)=0$. Thus:
		\begin{equation}\label{maxh^2+nablah^2}
			\max(h^2+|\nabla h|^2)\leq\max h^2.
		\end{equation}
		Combining \eqref{C0 estimate in C2 estimate}, there exists a constant $\tau_1$ such that
		\begin{equation}\label{C1 estimate}
			|\nabla h|\leq\tau_1.
		\end{equation}
		If $\det(h_{ij}+h\delta_{ij})=0$, consider $h_\epsilon=h+\epsilon$ such that $\det((h_\epsilon)_{ij}+h_\epsilon\delta_{ij})\neq0$ for any small enough $\epsilon>0$. Then \eqref{C1 estimate} is established for $h_\epsilon$. Let $\epsilon\rightarrow 0^+$, we have \eqref{C1 estimate} under two conditions. Then by \eqref{C0 estimate in C2 estimate}, \eqref{C1 estimate} and the equation \eqref{equation q}, we have
		\begin{equation}\label{determinant estimate}
			\frac{1}{\tau'}<\det(h_{ij}+h\delta_{ij})<\tau'
		\end{equation}
		for some positive constants $\tau'$. We claim that for any $e\in\sn$, the function $\nu(x)$: $e^\bot\rightarrow\mathbb{R}$ defined as
		\[
		\nu(x)=h(x+e)
		\]
		is a solution to
		\begin{equation}\label{transfer equation}
			\det(\nu_{ij}(x))=\frac{Z(\alpha,q)}{|x+e|^{n+p}}\nu^{p-1}(1-\frac{q}{\alpha}[|D\nu|^2+(\nu-x\cdot D\nu)^2]^{\frac{\alpha}{2}})^{\frac{n}{\alpha}+1-\frac{1}{q}}f(\frac{x+e}{\sqrt{1+|x|^2}})
		\end{equation}
		on $e^\bot$, where $h$ is the extending the support function from $\sn$ to $\rn$. For more particulars to this claim, see Lemma 3.1 in \cite{f23} or $C^2$ estimate in Theorem 4.1 in \cite{Ch19}. Combining \eqref{determinant estimate} and \eqref{transfer equation}, we can use the Caffarelli's result \cite{C89} to get that $\nu$ is $C^1$. Since the condition $f\in C^{\beta}(\sn)$, we can derive the $C^{2,\beta}$ estimate by using the Caffarelli's Schauder estimate in \cite{C90}.
	\end{proof}
	
	The next lemma is to show that the linearization operator is invertible when $p\geq n$.
	\begin{lem}\label{invertability of linearization operator}
		For $\beta\in(0,1)$, $p\geq n$, $\alpha>0$ and $q<\frac{\alpha}{n+\alpha}$, suppose that $f\in C_+^\beta(S^{n-1})$. If $h$ is a positive function which solves the equation \eqref{equation q}, then the linearized operator of \eqref{equation q} at $h$ is invertible.
	\end{lem}
	\begin{proof}
		Set $h_\epsilon=he^{\epsilon\zeta}$ for any $\zeta\in C^{2,\beta}(S^{n-1})$. Then by direct calculation we get the linearized operator of \eqref{equation q} at $h$:
		\begin{align*}
			L_h(\zeta)&=\frac{d}{d\epsilon}\det((h_\epsilon)_{ij}+h_\epsilon\delta_{ij})|_{\epsilon=0}-Z(\alpha,q)f\frac{d}{d\epsilon}[h_\epsilon^{p-1}(1-\frac{q}{\alpha}(h_\epsilon^2+|\nabla h_\epsilon|^2)^{\frac{\alpha}{2}})^{\frac{n}{\alpha}+1-\frac{1}{q}}]|_{\epsilon=0}\\
			&=L_h(1)\zeta+\omega^{ij}(h_i\zeta_j+h_j\zeta_i+h\zeta_{ij})\\
			&-Z(\alpha,q)(1-q-\frac{nq}{\alpha})fh^p[1-\frac{q}{\alpha}(h^2+|\nabla h|^2)^{\frac{\alpha}{2}}]^{\frac{n}{\alpha}-\frac{1}{q}}(h^2+|\nabla h|^2)^{\frac{\alpha}{2}-1}h_i\zeta_i
		\end{align*}
		where $\omega^{ij}$ is the cofactor of $(h_{ij}+h\delta_{ij})$, and
		\begin{align*}
			L_h(1)&=\omega^{ij}(h_{ij}+h\delta_{ij})-Z(\alpha,q)f\{(p-1)h^{p-1}[1-\frac{q}{\alpha}(h^2+|\nabla h|^2)^{\frac{\alpha}{2}}]^{\frac{n}{\alpha}+1-\frac{1}{q}}\\
			&-(1-q-\frac{nq}{\alpha})h^{p-1}[1-\frac{q}{\alpha}(h^2+|\nabla h|^2)^{\frac{\alpha}{2}}]^{\frac{n}{\alpha}-\frac{1}{q}}(h^2+|\nabla h|^2)^{\frac{\alpha}{2}}\}\\
			&=Z(\alpha,q)h^{p-1}[1-\frac{q}{\alpha}(h^2+|\nabla h|^2)^{\frac{\alpha}{2}}]^{\frac{n}{\alpha}+1-\frac{1}{q}}f[n-p-\frac{(1-q-\frac{nq}{\alpha})(h^2+|\nabla h|^2)^{\frac{\alpha}{2}}}{1-\frac{q}{\alpha}(h^2+|\nabla h|^2)^{\frac{\alpha}{2}}}]\\
			&<0
		\end{align*}
		provided $q<\frac{\alpha}{n+\alpha}$. One thing needs to be ensured: if $L_h(\zeta)=0$, then $\zeta=0$. Indeed, if $L_h(\zeta)=0$, then
		\begin{align*}
			-L_h(1)\zeta&=\omega^{ij}(h_i\zeta_j+h_j\zeta_i+h\zeta_{ij})\\
			&-Z(\alpha,q)(1-q-\frac{nq}{\alpha})fh^p[1-\frac{q}{\alpha}\gamma(h^2+|\nabla h|^2)^{\frac{\alpha}{2}}]^{\frac{n}{\alpha}-\frac{1}{q}}\alpha\gamma(h^2+|\nabla h|^2)^{\frac{\alpha}{2}-1}h_i\zeta_i.
		\end{align*}
		Suppose $\zeta$ attains its maximum at $u_0$, then: $\nabla\zeta|_{u_0}=0$ and $\nabla^2\zeta|_{u_0}\leq0$. Combining with the fact that $\omega^{ij}$ is positive definite, we have
		\[
		-L_h(1)\zeta_{\max}=\omega^{ij}\zeta_{ij}h\leq0,
		\]
		which means $\zeta_{max}\leq0$. Similarly, consider the minimum of $\zeta$, we can also get $\zeta_{min}\geq0$. Then
		\[
		\zeta\equiv0.
		\]
		Then we finish the proof.
	\end{proof}
	
	We are now ready to prove the existence part of Theorem \ref{smooth solution for p>n1} by using the continuity method.
	
	\begin{proof}[Proof of Theorem \ref{smooth solution for p>n1} for Existence]
		We use the continuity method to prove the existence. Consider a family of equation
		\begin{equation}\label{family equation in continuity method}
			\det(h_{ij}+h\delta_{ij})=Z(\alpha,q)h^{p-1}g_{\alpha,q}^{-1}(|Dh|)f_t,
		\end{equation}
		where $f_t=(1-t)c_0+tf$ and $c_0$ is a constant such that $\frac{1}{C}<c_0,f<C$ for some positive constant $C$. Then $\frac{1}{C}<f_t<C$. Set
		\[
		I=\{t\in[0,1]:\eqref{family equation in continuity method}\ \text{admits a solution}\ h_{K_t}\in C^{2,\beta}(\sn)\}.
		\]
		We will show that $I$ is both open and closed, and a non-empty set, which implies $I=[0,1]$. That implies $t=1$ admits a solution to equation \eqref{equation q}.
		
		Clearly, when $t=0$, equation \eqref{equation q} must have a constant solution. That means $0\in I$, and $I$ is not an empty set. Next, the openness is followed immediately by Lemma \ref{invertability of linearization operator} and the Implicit Function Theorem. Finally, the closeness is followed by Lemma \ref{C0 estimate for p>n} and Lemma \ref{$C^2$ estimate}. Thus, $I=[0,1]$.
	\end{proof}

	Moreover, for $p\geq n$, we can derive the uniqueness by maximum principle, which include the  uniqueness part of Theorem \ref{smooth solution for p>n1}.
	
	\begin{thm}[\bf Uniqueness]For $p\geq n$, $\alpha>0$ and $q<\frac{\alpha}{n+\alpha}$, the equation \eqref{smooth solution for p>n1} has a uniqueness solution.
	\end{thm}
	\begin{proof}
		 Let $h_1$ and $h_2$ be two solutions of \eqref{smooth solution for p>n1} and $G=\frac{h_1}{h_2}$. Assume $G(x_0)=G_{\max},$ then at $x_0,$
		$$
		0=\nabla G=\frac{(\nabla h_1)h_2-h_1\nabla h_2}{h_2^2}
		$$
		and
		$$
		0\geq\nabla^2G=\frac{h_2\nabla^2h_1-h_1\nabla^2h_2}{h_2^2},
		$$
		i.e.,
		$$
		\frac{\nabla^2h_1}{h_1}\leq\frac{\nabla^2h_2}{h_2}.
		$$
		Hence
		\begin{equation*}
			\begin{split}
				fh_1^{p-1}[1-\frac{q}{\alpha}(|\nabla h_1|^2+h_1^2)^{\frac{\alpha}{2}}]^{1+\frac{n}{\alpha}-\frac{1}{q}}&=h_1^{n-1}\det\left(\frac{\nabla^2h_1}{h_1}+I\right)\\
				&\leq h_1^{n-1}\det\left(\frac{\nabla^2h_2}{h_2}+I\right)\\
				&=\frac{h_1^{n-1}}{h_2^{n-1}}fh_2^{p-1}[1-\frac{q}{\alpha}(|\nabla h_2|^2+h_2^2)^{\frac{\alpha}{2}}]^{1+\frac{n}{\alpha}-\frac{1}{q}},
			\end{split}
		\end{equation*}
		i.e.,
		$$
		h_1^{p-n}[1-\frac{q}{\alpha}(|\nabla h_1|^2+h_1^2)^{\frac{\alpha}{2}}]^{1+\frac{n}{\alpha}-\frac{1}{q}}\leq h_2^{p-n}[1-\frac{q}{\alpha}|(|\nabla h_2|^2+h_2^2)^{\frac{\alpha}{2}}]^{1+\frac{n}{\alpha}-\frac{1}{q}}.
		$$
		Set $\frac{|\nabla h_1|}{h_1}=\frac{|\nabla h_2|}{h_2}=c,$ then
		$$
		h_1^{p-n}[1-\frac{q}{\alpha}(1+c^2)^{\frac{\alpha}{2}}h_1^\alpha]^{1+\frac{n}{\alpha}-\frac{1}{q}}\leq h_2^{p-n}[1-\frac{q}{\alpha}(1+c^2)^{\frac{\alpha}{2}}h_2^\alpha]^{1+\frac{n}{\alpha}-\frac{1}{q}},
		$$
		the function $t^{p-n}[1-\frac{q}{\alpha}(1+c^2)^{\frac{\alpha}{2}}t^\alpha]^{1+\frac{n}{\alpha}-\frac{1}{q}}$ is of $t$ increasing, it follows from
		\begin{small}
			\begin{equation*}
				\begin{split}
					\left(t^{p-n}[1-\frac{q}{\alpha}(1+c^2)^{\frac{\alpha}{2}}t^\alpha]^{1+\frac{n}{\alpha}-\frac{1}{q}}\right)'&=(p-n)t^{p-n-1}[1-\frac{q}{\alpha}(1+c^2)^{\frac{\alpha}{2}}t^\alpha]^{1+\frac{n}{\alpha}-\frac{1}{q}}\\
					&+(1-q-\frac{qn}{\alpha})t^{p-n+\alpha-1}[1-\frac{q}{\alpha}(1+c^2)^{\frac{\alpha}{2}}t^\alpha]^{\frac{n}{\alpha}-\frac{1}{q}}(1+c^2)^{\frac{\alpha}{2}}\\&\geq 0
				\end{split}
			\end{equation*}
		\end{small}
		for $p\geq n$ and $q<\frac{\alpha}{n+\alpha}$. Hence
		$$
		h_1(x_0)\leq h_2(x_0),
		$$
		i.e.,
		$$
		h_1(x)\leq h_2(x).
		$$
		Similarly we get that $h_1(x)\geq h_2(x)$. Thus $h_1(x)=h_2(x)$.
	\end{proof}

	\section{Existence of Smooth Solution for $1\leq p<n$}\label{degree theory}
	In this section, we use degree theory to derive a smooth solution to the  Monge-Amp\`ere equation of the $L_p$ generalized Gaussian Minkowski problem for $1\leq p<n$. The approach introduced by Chen-Hu-Liu-Zhao \cite{CHLZ} inspires us to find two types of solutions. One satisfies $G_{\alpha,q}([K_1])<\frac{1}{2}$ with $K_1\in\kne$,  the other  satisfies $G_{\alpha,q}([K_2])>\frac{1}{2}$ with $K_2\in\kno$.
	
	We would like to emphasize that within the framework of degree theory, the uniqueness of solutions to the associated isotropic Monge-Amp\`ere equation stands as one of the key parts. Very recently, Ivaki takes a step toward to the uniqueness of solutions to a class of isotropic curvature problem in \cite{I23} as follows.
	\begin{thm}\cite{I23} Suppose $\varphi:(0,\infty)\times(0,\infty)\rightarrow(0,\infty)$ is $C^1$-smooth with $\partial_1 \varphi \geq0,\partial_2 \varphi \geq0$ and at least one of these inequalities is strict. If
		$M^n$
		is a closed, smooth, strictly convex hypersurface with the support
		function $h>0$ and Gauss curvature $\kappa$, such that
		$$
		\varphi(h,|Dh|)\kappa=1,
		$$
		then $M^n$
		is a rescaling of $\sn$. Moreover, if $\partial_1 \varphi =0$, then the same
		conclusion holds under no extra assumption on the sign of $h$.
	\end{thm}
	
	Consider taking $\varphi(h,|Dh|)=ch^{p-1}g^{-1}_{\alpha,q}(|Dh|)$, we have the following Theorem \ref{uniqueness of isoperimetric equation}. It is worth noting that the lower bound of $h$ will be shown in the a priori estimate in Lemma \ref{$C^0$ estimate for 1<p<n}. Thus the theorem becomes
	
	\begin{thm}\label{uniqueness of isoperimetric equation}
		For $1\leq p<n$, $\alpha>0$ and $q<\frac{\alpha}{n+\alpha}$, if $h$ is a positive solution to the equation
		\begin{equation}\label{a}
			ch^{p-1}g^{-1}_{\alpha,q}(|Dh|)\kappa=1,
		\end{equation}
		where $c$ is a positive constant. Then $h$ has to be a constant. Moreover,
		
		(i) If  $c\in (0,Z^{-1}(\alpha,q)(\frac{n-p}{q(\frac{1}{q}-\frac{n}{\alpha}-1)+n-p})^{\frac{n-p}{\alpha}}(1-\frac{q}{\alpha}\frac{n-p}{q(\frac{1}{q}-\frac{n}{\alpha}-1)+n-p})^{\frac{1}{q}-\frac{n}{\alpha}-1})$, equation \eqref{a} has precisely two constant solutions.
		
		(ii) If  $c=Z^{-1}(\alpha,q)(\frac{n-p}{q(\frac{1}{q}-\frac{n}{\alpha}-1)+n-p})^{\frac{n-p}{\alpha}}(1-\frac{q}{\alpha}\frac{n-p}{q(\frac{1}{q}-\frac{n}{\alpha}-1)+n-p})^{\frac{1}{q}-\frac{n}{\alpha}-1}$, equation \eqref{a} has a uniqueness constant solution.

		(iii) If  $c\in (Z^{-1}(\alpha,q)(\frac{n-p}{q(\frac{1}{q}-\frac{n}{\alpha}-1)+n-p})^{\frac{n-p}{\alpha}}(1-\frac{q}{\alpha}\frac{n-p}{q(\frac{1}{q}-\frac{n}{\alpha}-1)+n-p})^{\frac{1}{q}-\frac{n}{\alpha}-1},\infty)$, equation \eqref{a} has no solution.
	\end{thm}
	
	The degree theory also need some a priori estimates to the equation \eqref{equation q}.
	
	\begin{lem}\textbf{($C^0$ estimate)}\label{$C^0$ estimate for 1<p<n}
		For $0<p<n$, $\alpha>0$ and $q<\frac{\alpha}{n+\alpha}$, suppose $K\in\kno$(resp., $K\in\kne$) such that $G_{\alpha,q}(K)\geq\frac{1}{2}$(resp., $G_{\alpha,q}(K)<\frac{1}{2}$) and its support function $h$ satisfies the equation \eqref{equation q}. If there exists a constant $\tau>0$, such that $\frac{1}{\tau}<f<\tau$, then there exists a constant $\tau'>0$, depend only on $\tau$, such that
		\[
		\frac{1}{\tau'}<h<\tau'.
		\]
	\end{lem}
	\begin{proof}
		The upper bound of the support function has been given by Lemma \ref{C0 estimate for p>n}. We only need to provide the positive lower bound.\\
		\textbf{Case(i)}: For $K\in\kno$ and $G_{\alpha,q}(K)\geq\frac{1}{2}$, we provide its lower bound by a contradiction. If $K$ does not have a lower bound, then there exists $v\in\sn$ such that $h(v)=0$. Then for any $\epsilon>0$, $K\subset\{x|x\cdot v>-\epsilon\}:=H_\epsilon$. Since $h$ is bounded from above, there exists a large radius $R>0$ such that
		\[
		K\subset B_R\cap H_\epsilon.
		\]
		Denote $H:=\lim_{\epsilon\rightarrow0}H_\epsilon$. Thus
		\begin{align*}
			\frac{1}{2}&=\int_{H}g_{\alpha,q}(x)dx\\
			&=\int_{H\setminus B_R}g_{\alpha,q}(x)dx+\int_{H\cap B_R }g_{\alpha,q}(x)dx\\
			&>\int_{H\cap B_R }g_{\alpha,q}(x)dx\\
			&\geq G_{\alpha,q}(K),
		\end{align*}
		which is a contradiction to $G_{\alpha,q}(K)\geq\frac{1}{2}$.\\
		\textbf{Case(ii)}:  For $K\in\kne$ and $G_{\alpha,q}(K)<\frac{1}{2}$, suppose that $h$ attains its maximum at $v_0$. By equation \eqref{equation q}, we have
		\begin{align*}
			h^{n-1}_{\max}&\geq \det(h_{ij}+h\delta_{ij})\\
			&=Z(\alpha,q)h_{\max}^{p-1}[1-\frac{q}{\alpha} h_{\max}^\alpha]^{1+\frac{n}{\alpha}-\frac{1}{q}}f \\
			&\geq Z(\alpha,q)(\inf f)h_{\max}^{p-1}.
		\end{align*}
		Combining $p<n$ and $f>\frac{1}{\tau}$, there exists a positive constant $\tau_0$ such that
		\begin{equation}\label{upper bound to hmax}
			h_{\max}>\tau_0.
		\end{equation}

We next provide the lower bound by making good use of the upper bound of $h$ and the lower bound of $h_{\max}$. Assume that $h$ attains its minimum at $u_0$ on $\sn$. By the equation \eqref{equation q}, we have
		\begin{align*}
			\frac{1}{n}h\det(h_{ij}+h\delta_{ij})&=\frac{1}{n}h^pg^{-1}_{\alpha,q}(|Dh|)f\\
			&\geq Ch^p,
		\end{align*}
		for which the total integral of the left-hand side on $\sn$ is the volume of $K$. Therefore
		\begin{equation}\label{total integral}
			\mathcal{H}^n(K)\geq C\int_{\sn}h^p(v)dv.
		\end{equation}
		Note that $h$ is an even solution, then by the definition of support function,
		\begin{equation}\label{support function inequality}
			h(v)\geq h_{\max}|v\cdot v_0|.
		\end{equation}
		Combining \eqref{upper bound to hmax}, \eqref{total integral}, \eqref{support function inequality}, we have
		\begin{equation}\label{lower bound of volume}
			\mathcal{H}^n(K)\geq Ch_{\max}^p\int_{\sn}|v\cdot v_0|^pdv=C'.
		\end{equation}
		Also note the truth
		\[
		K\subset(h_{\max}B_1)\cap\{x\in\rn:|x\cdot u_0|\leq h_{\min}\},
		\]
		where $B_1$ denote the ball with radiu 1 centered at origin, and it implies
		\begin{equation}\label{upper bound of volume}
			\mathcal{H}^n([h])\leq2^n h_{\max}^{n-1}h_{\min}<C h_{\min}.
		\end{equation}
		Combining \eqref{lower bound of volume} and \eqref{upper bound of volume}, $h_{\min}$ such that
		\[
		h_{min}>\tau'>0.
		\]	\end{proof}
	
	\begin{lem}\textbf{$(C^2$ estimate$)$}\label{C2 estimate for 1<p<n}
		For $\beta\in(0,1)$, $0<p<n$, $\alpha>0$ and $q<\frac{\alpha}{n+\alpha}$, suppose $K\in\kno$$($resp., $K\in\kne$$)$ such that $G_{\alpha,q}(K)\geq\frac{1}{2}$$($resp., $G_{\alpha,q}(K)<\frac{1}{2}$$)$ and its support function $h\in C^{2,\beta}(\sn)$ satisfies the equation \eqref{equation q}. If $f\in C^\beta(\sn)$ and there exists a constant $\tau>0$ such that $\frac{1}{\tau}<f<\tau$,
		then there exists a constant $\tau'$ which only depend on $\tau$ such that
		\[
		|h|_{C^{2,\beta}}<\tau'.
		\]
	\end{lem}
	\begin{proof}
		Based on Lemma \ref{$C^0$ estimate for 1<p<n}, we can derive the desire result through a similarly proof as in Lemma \ref{$C^2$ estimate}.
	\end{proof}
	
	We are now ready to  prove the existence.
	
	\begin{thm}\label{sommth solution for 1<=p<n}
		For $\beta\in(0,1)$ and $1\leq p<n$, suppose that $0\leq q<\frac{\alpha}{n+\alpha}$ with $\alpha>0$ or $q<0$ with $0<\alpha\leq1$. If $f\in C^{\beta}(\sn)$ is a positive function and satisfies $\|f\|_{L_1}<(\frac{n}{2})^{1-p}I^p[G_{\alpha,q}](\frac{1}{2})$, where $I[G_{\alpha,q}]$ is the isoperimetric profile, then there exist two $C^{2,\beta}$ convex bodies $K_1\in \kne$ and $K_2\in \kno$, whose support functions $h_{K_1}$ and $h_{K_2}$ satisfy equation \eqref{equation q} with $G_{\alpha,q}(K_1)<\frac{1}{2}$ and $G_{\alpha,q}(K_2)>\frac{1}{2}$.
	\end{thm}
	
	\begin{proof}
		 Due to Theorem \ref{uniqueness of isoperimetric equation}, we require that $c$ in equation \eqref{a} is small enough so that \eqref{a} has exactly two constant solutions $h\equiv r_1$ and $h\equiv r_2$ with $0<r_1<r_2$. It is clear that $G_{\alpha,q}([r_1])<\frac{1}{2}$ and $G_{\alpha,q}([r_2])>\frac{1}{2}$ for some $c$ small enough. We also require that $c$ is small enough, so it such that $|c|_{L_1}<(\frac{n}{2})^{1-p}I^p[G_{\alpha,q}](\frac{1}{2})$ for $0\leq q<\frac{\alpha}{n+\alpha}$ with $\alpha>0$ or $q<0$ with $0<\alpha\leq1$.  Finally, we also need to make sure that the linearized operator of the equation \eqref{a}
		\begin{equation*}
			\begin{split}
				L_{r_i}&={r_i}^{n-2}\triangle_{\sn}\phi+(n-1){r_i}^{n-2}\phi-\frac{(\alpha q+nq-\alpha)r_i^{n-1+\alpha-p}}{\alpha-qr_i^\alpha}\phi-(p-1){r_i}^{n-2}\phi\\
				&={r_i}^{n-2}(\triangle_{\sn}\phi+((n-p)+\frac{(\alpha q+nq-\alpha)r_i^{\alpha+1-p}}{\alpha-qr_i^\alpha})\phi)
			\end{split}
		\end{equation*}
		is invertible for $i=1,2$, where we can do it by selecting some $c>0$.
		
		Define a family of operators $F_t:C^{2,\beta}(\sn)\rightarrow C^\beta(\sn)$ by
		\begin{equation*}
			F_t=\det(\nabla^2 h+hI)-g^{-1}_{\alpha,q}(|Dh|)h^{p-1}f_t,
		\end{equation*}
		where $f_t=(1-t)c+tf$ for some $t\in[0,1]$. Since $f\in C^\beta(\sn)$, there exists a constant $\tau>0$ such that $\frac{1}{\tau}<f,c<\tau$ and $|f|_{C^\beta}<\tau$. Then for each $t\in[0,1]$, $f_t$ has the same bound as $f$, i.e., $\frac{1}{\tau}<f_t<\tau$, $|f_t|_{L^1}<(\frac{n}{2})^{1-p}I^p[G_{\alpha,q}](\frac{1}{2})$, and $|f_t|_{C^\beta}<\tau$. Then set $\tau'>0$ be the constant in Lemma \ref{$C^0$ estimate for 1<p<n} and Lemma $\ref{C2 estimate for 1<p<n}$.
		
		Next, define two open bounded set $O_1,O_2\subset C^{2,\beta}(\sn)$ by
		\[
		O_1=\{h\in C_e^{2,\beta}(\sn): \frac{1}{\tau'}<h<\tau', |h|_{C^{2,\beta}}<\tau', G_{\alpha,q}([h])<\frac{1}{2}\}
		\]
		and
		\[
		O_2=\{h\in C^{2,\beta}(\sn): \frac{1}{\tau'}<h<\tau', |h|_{C^{2,\beta}}<\tau', G_{\alpha,q}([h])>\frac{1}{2}\}.
		\]
		We claim that
		\begin{equation}\label{well define condition for degree}
			\partial O_i\cap F_t^{-1}(0)=\emptyset
		\end{equation}
		for $i=1,2$ and $t\in[0,1]$. Indeed, if $h\in \partial O_i\cap F_t^{-1}(0)$, then $h$ solves
		\begin{equation}\label{monge ampere equation for t}
			h^{1-p}g_{\alpha,q}(|Dh|)\det(\nabla^2h+hI)=f_t,
		\end{equation}
		 $G_{\alpha,q}([h])\leq\frac{1}{2}$ for $i=1$ and $G_{\alpha,q}([h])\geq\frac{1}{2}$ for $i=2$. But $G_{\alpha,q}([h])$ cannot be less than $\frac{1}{2}$ for $i=1$ and more than $\frac{1}{2}$ for $i=2$. This is because Lemma \ref{$C^0$ estimate for 1<p<n} and Lemma \ref{C2 estimate for 1<p<n} show that $h$ satisfies
		\[
		\frac{1}{\tau'}<h<\tau', |h|_{C^{2,\beta}}<\tau',
		\]
		which contradicts the openness of $O_i$. Thus, if $h\in\partial O_i\cap F_t^{-1}(0)$, then $G_{\alpha,q}([h])=\frac{1}{2}$. But  the isoperimetric inequalities in Corollary \ref{isoperimetric inequalities for index} show that
		\[
		|S_{p,\alpha,q}([h])|\geq (\frac{n}{2})^{1-p}I^p[G_{\alpha,q}](\frac{1}{2})>0,
		\]
		which contradicts the condition $|f|_{L^1}<(\frac{n}{2})^{1-p}I^p[G_{\alpha,q}](\frac{1}{2})$. That proves \eqref{well define condition for degree}, and makes sure that the degree $\deg(F_t,O_i,0)$ is well-defined for $t\in[0,1]$ and not depend on $t$. Combining the Proposition 2.3 in \cite{Li89} and Theorem \ref{uniqueness of isoperimetric equation}, we have
		\[
		\deg(F_0,O_i,0)=\deg(L_{r_0},O_i,0)\neq0,
		\]
		where the last inequality follows from the Proposition 2.4 in \cite{Li89}. By proposition 2.2 in \cite{Li89},
		\[
		\deg(F_1,O_i,0)=\deg(F_0,O_i,0)\neq0
		\]
		for $i=1,2$, which implies the existence of $h\in O_i$ such that $F_1(h)=0$.
	\end{proof}
    \begin{rem}
	In an up-coming work \cite{LLT} concerning the Minkowski problem for a class of ratational invariant measures, Theorem \ref{sommth solution for 1<=p<n} is extended to the range $\alpha>1$ with $q<0$, which means that there exist two solutions to the equation \eqref{equation q} in the range $\alpha>0$ with $q<\frac{\alpha}{n+\alpha}$.
    \end{rem}
	
	\section{Approximation}\label{Approximation}
	We remove the regularity assumption in Section \ref{continuity method for $p>n$} and Section \ref{degree theory} by approximation arguments.
	
	\begin{proof}[Proof of Theorem \ref{approximation argument for p>n}]
		Let $d\mu_i=f_id\nu$ be a sequence of measures which weakly converges to $d\mu=fd\nu$, where $f_i$ is smooth enough. Since $\frac{1}{C}<f<C$, we can moreover suppose that $\frac{1}{C}<f_i<C$ by selecting a subsequence. Thus by Theorem \ref{smooth solution for p>n1}, for every $f_i$ there exists a convex body $K_i$, its support function $h_{K_i}$ satisfies the equation \eqref{equation q}. Then by a priori estimate in Lemma \ref{C0 estimate for p>n}, $h_{K_i}$ has the upper and lower bound, i.e., there exists a positive constant $C'$ such that
		\[
		\frac{1}{C'}B\subset K_i\subset C'B.
		\]
		By the upper bound and Blaschke's selection theorem, $K_i$ converges to a convex body $K$ by selecting a subsequence. And by the lower bound, $K\in\kno$.
		
		Finally, since $K_i\rightarrow K$ and weakly convergence of $S_{p,\alpha,q}$ with respect to Hausdorff metric in Lemma \ref{weakly convergent wrt Haussdorff metric},
		\[
		S_{p,\alpha,q}(K,\cdot)=\mu.
		\]
	\end{proof}	
	
	\begin{proof}[Proof of Theorem \ref{approximation argument for 1<=p<n}]
		Let $d\mu_i=f_id\nu$ be a sequence of measures which weakly converges to $\mu$, $0<f_i\in C^{\beta}$ with $\|f\|_{L_1}<(\frac{n}{2})^{1-p}I^p[G_{\alpha,q}](\frac{1}{2})$. Note that these conditions are the same as Theorem \ref{sommth solution for 1<=p<n}, then there exists a $C^{2,\alpha}$ convex body $K_i\in\kno$ with $G_{\alpha,q}(K_i)>\frac{1}{2}$, its support function $h$ satisfies
		\[
		h^{1-p}g_{\alpha,q}(|Dh|)\det(\nabla^2h+hI)=f_i,
		\]
		which means $dS_{p,\alpha,q}(K_i,\cdot)=d\mu_i$. We claim that $K_i$ is uniformly bounded from above and below. If this do, we can select a convergent subsequence $K_i$ converges to $K\in\kno$.
		
		We give the upper bound by a contradiction. Since $\mu$ is not concentrated on any closed hemisphere, then by Lemma 5.4 in \cite{L22}, there exist $\epsilon_0>0$, $\delta_0>0$, and $N_0>0$ such that for any $i>N_0$, $e\in\sn$,
		\begin{equation}\label{not concentrated on any closed hemisphere}
			\int_{\sn\cap\{v|v\cdot e>\delta_0\}}f_id\nu>\epsilon.
		\end{equation}
		Suppose that $h_{K_i}$ attains its maximum at $v_i\in\sn$. Recall the definition of support function, we have the classical inequality
		\begin{equation}\label{classical inequality of support function}
			h_{K_i}(v)\geq(v\cdot v_i)\rho_{K_i}(v_i)=(v\cdot v_i)h_{K_i}(v_i).
		\end{equation}
		Combining \eqref{not concentrated on any closed hemisphere} and \eqref{classical inequality of support function}, we have
		\begin{small}
			\begin{align*}
				\epsilon&<\int_{\sn\cap\{v|v\cdot v_i>\delta_0\}}f_id\nu\\
				&=\frac{1}{Z(\alpha,q)}\int_{\sn\cap\{v|v\cdot v_i>\delta_0\}}h_{K_i}^{1-p}[1-\frac{q}{\alpha}|h_{K_i}+\nabla h_{K_i}^2|^{\frac{\alpha}{2}}]^{\frac{1}{q}-\frac{n}{\alpha}-1}\det(\nabla^2h_{K_i}+h_{K_i}I)d\nu\\
				&\leq\frac{1}{Z(\alpha,q)}(h_{K_i}(v_i)\delta_0)^{n-p}\delta_0^{-p}[1-\frac{q}{\alpha}(\delta h_{K_i}(v_i)^\alpha)]^{\frac{1}{q}-\frac{n}{\alpha}-1}\mathcal{H}(\sn).
			\end{align*}
		\end{small}
		If $h_{K_i}$ is not bounded from above, then $h_{K_i}(v_i)\rightarrow\infty$, which means the above right end  converges to $0$. It is a  contradiction. And for the positive lower bound of $h_{K_i}$, the proof is the same as the proof of lower bound in \textbf{CASE(i)} in Lemma \ref{$C^0$ estimate for 1<p<n}. Thus we finish the proof.		
	\end{proof}
	
	\begin{proof}[Proof of Theorem \ref{small volume approximation argument for 1<=p<n}]
		It is the same as the proof of Theorem \ref{approximation argument for p>n}, and what we only need to prove is that $G_{\alpha,q}(K)<\frac{1}{2}$. It is easy to prove when combining the continuity of generalized Gaussian volume $G_{\alpha,q}$ and the isoperimetric inequality in Corollary \ref{isoperimetric inequalities for index} when $G_{\alpha,q}(K)=\frac{1}{2}$.
	\end{proof}

	\section*{Acknowledgement}
	The authors are very grateful Prof. G. Zhang for introducing this problem. We would also like to thank Prof. Y. Huang and Dr. J. Hu for their patient guidance and advices.
	
	\section*{Statements and Declarations}
	There is no conflict of interest in this work.


\begin{thebibliography}{99}
		
		\bibitem{A38}
		A.D. Aleksandrov, On the theory of mixed volumes. III. Extensions of two theorems of Minkowski on convex polyhedra to arbitrary convex bodies, Mat. Sb. (N.S.) 3 (1938), 27--46.
		\bibitem{BBCY19}
		G. Bianchi, K. B\"{o}r\"{o}czky, A. Colesanti\ and\ D. Yang, The $L_p$-Minkowski problem for $-n<p<1$, Adv. Math. {\bf 341} (2019), 493--535.
		\bibitem{B07}
		S.~G. Bobkov, Large deviations and isoperimetry over convex probability measures with heavy tails, Electron. J. Probab. {\bf 12} (2007), 1072--1100.
		\bibitem{B74}
		C. Borell, Convex measures on locally convex spaces, Ark. Mat. {\bf 12} (1974), 239--252.
		\bibitem{B75}
		C. Borell, Convex set functions in $d$-space, Period. Math. Hungar. {\bf 6} (1975), no.~2, 111--136.
		\bibitem{Bo13}
		K. B\"{o}r\"{o}czky, E. Lutwak, D. Yang\ and\ G. Zhang, The logarithmic Minkowski problem, J. Amer. Math. Soc. {\bf 26} (2013), no.~3, 831--852.
		
		\bibitem{C89}
		L. Caffarelli, Interior a priori estimates for solutions of fully nonlinear equations, Ann. of Math. (2) {\bf 130} (1989), no.~1, 189--213.
		\bibitem{C90}
		L. Caffarelli, A localization property of viscosity solutions to the Monge-Amp\`ere equation and their strict convexity, Ann. of Math. (2) {\bf 131} (1990), no.~1, 129--134.
		\bibitem{CFGLSW}
		U. Caglar, M. Fradelizi, O. Gu\'{e}don, J. Lehec, C. Sch\"{u}tt\ and\ E. M. Werner, Functional versions of $L_p$-affine surface area and entropy inequalities, Int. Math. Res. Not. IMRN {\bf 2016}, no.~4, 1223--1250.
		\bibitem{G16}
		G.~R. Chambers, Proof of the log-convex density conjecture, J. Eur. Math. Soc. (JEMS) {\bf 21} (2019), no.~8, 2301--2332.
		\bibitem{CHLZ}
		S. Chen, S. Hu, W. Liu\ and\ Y. Zhao, On the planar Gaussian-Minkowski problem, Adv. Math. {\bf 435} (2023), part A, Paper No. 109351, 32 pp.
		\bibitem{CW06}
		K.~S. Chou\ and\ X.~J. Wang, The $L_p$-Minkowski problem and the Minkowski problem in centroaffine geometry, Adv. Math. {\bf 205} (2006), no.~1, 33--83.
		
		\bibitem{C23}
		D. Cordero-Erausquin\ and\ L. Rotem, Improved log-concavity for rotationally invariant measures of symmetric convex sets, Ann. Probab. {\bf 51} (2023), no.~3, 987--1003.
		\bibitem{CC84}
		M.~H.~M. Costa\ and\ T.~M. Cover, On the similarity of the entropy power inequality and the Brunn-Minkowski inequality, IEEE Trans. Inform. Theory {\bf 30} (1984), no.~6, 837--839.	
		\bibitem{DD02}
		M. Del~Pino\ and\ J. Dolbeault, Best constants for Gagliardo-Nirenberg inequalities and applications to nonlinear diffusions, J. Math. Pures Appl. (9) {\bf 81} (2002), no.~9, 847--875.
		\bibitem{DD03}
		M. Del~Pino\ and\ J. Dolbeault, The optimal Euclidean $L^p$-Sobolev logarithmic inequality, J. Funct. Anal. {\bf 197} (2003), no.~1, 151--161.
		\bibitem{f23}
		Y. Feng, S. Hu\ and\ L. Xu, On the $L_p$ Gaussian Minkowski problem, J. Differential Equations {\bf 363} (2023), 350--390.
		\bibitem{FLX}
		Y. Feng, W. Liu\ and\ L. Xu, Existence of non-symmetric solutions to the Gaussian Minkowski problem, J. Geom. Anal. {\bf 33} (2023), no.~3, Paper No. 89, 39 pp.
		\bibitem{FLMZ}
		M. Fradelizi, D. Langharst, M. Madiman\ and\  A. Zvavitch, Weighted Brunn-Minkowski theory I: On weighted surface area measures, J. Math. Anal. Appl. {\bf 529} (2024), no.~2, Paper No. 127519, 30 pp.
		\bibitem{FLMZ2}
		M. Fradelizi, D. Langharst, M. Madiman\ and\  A. Zvavitch, Weighted Brunn-Minkowski Theory II: Inequalities for Mixed Measures and Applications, arXiv:2402.10314.
		\bibitem{GHW14}
		R.~J. Gardner, D. Hug\ and\ W. Weil, The Orlicz-Brunn-Minkowski theory: a general framework, additions, and inequalities, J. Differential Geom. {\bf 97} (2014), no.~3, 427--476.
		\bibitem{GHWXY19}
		R.~J. Gardner, D. Hug, W. Weil, S. Xing,\ and\ D. Ye, General volumes in the Orlicz-Brunn-Minkowski theory and a related Minkowski problem I, Calc. Var. Partial Differential Equations {\bf 58} (2019), no.~1, Paper No. 12, 35 pp.
		\bibitem{GHXY20}
		R.~J. Gardner, D. Hug, S. Xing,\ and\ D. Ye, General volumes in the Orlicz-Brunn-Minkowski theory and a related Minkowski problem II, Calc. Var. Partial Differential Equations {\bf 59} (2020), no.~1, Paper No. 15, 33 pp.
		\bibitem{GZ10}
		R.~J. Gardner\ and\ A. Zvavitch, Gaussian Brunn-Minkowski inequalities, Trans. Amer. Math. Soc. {\bf 362} (2010), no.~10, 5333--5353.
		\bibitem{Ha10}
		C. Haberl, E. Lutwak, D. Yang\ and\ G. Zhang, The even Orlicz Minkowski problem, Adv. Math. {\bf 224} (2010), no.~6, 2485--2510.
		\bibitem{Ha19}
		C. Haberl\ and\ F.~E. Schuster, Affine vs. Euclidean isoperimetric inequalities, Adv. Math. {\bf 356} (2019), 106811, 26 pp.
		\bibitem{H}
		J. Hu et al.,  The Minkowski problems in the generalized Gaussian space, preprint.	
		\bibitem{HH12}
		Q. Huang\ and\ B. He, On the Orlicz Minkowski problem for polytopes, Discrete Comput. Geom. {\bf 48} (2012), no.~2, 281--297.
		\bibitem{HJX15}
		Y. Huang, J. Liu\ and\ L. Xu, On the uniqueness of $L_p$-Minkowski problems: the constant $p$-curvature case in $\Bbb{R}^3$, Adv. Math. {\bf 281} (2015), 906--927.
		\bibitem{HLYZ}
		Y. Huang, E. Lutwak, D. Yang\ and\ G. Zhang, Geometric measures in the dual Brunn-Minkowski theory and their associated Minkowski problems, Acta Math. {\bf 216} (2016), no.~2, 325--388.
		\bibitem{HLYZ18}
		Y. Huang, E. Lutwak, D. Yang\ and\ G. Zhang, The $L_p$-Aleksandrov problem for $L_p$-integral curvature, J. Differential Geom. {\bf 110} (2018), no.~1, 1--29.
		\bibitem{H21}
		Y. Huang, D. Xi\ and\ Y. Zhao, The Minkowski problem in Gaussian probability space, Adv. Math. {\bf 385} (2021), Paper No. 107769, 36 pp.
		\bibitem{Ch19}
		Y. Huang\ and\ Y. Zhao, On the $L_p$ dual Minkowski problem, Adv. Math. {\bf 332} (2018), 57--84.
			\bibitem{N}
		N. Huet, Isoperimetry for spherically symmetric log-concave probability measures, Rev. Mat. Iberoam. {\bf 27} (2011), no.~1, 93--122.
		\bibitem{I23}
		M.~N. Ivaki, Uniqueness of solution to a class of non-homogeneous curvature problems, arXiv:2307.06252v2.
		
		\bibitem{IM23}
		M.~N. Ivaki\ and\ E. Milman, Uniqueness of solutions to a class of isotropic curvature problems, Adv. Math. {\bf 435} (2023), part A, Paper No. 109350, 11 pp.
		\bibitem{JL19}
		H. Jian\ and\ J. Lu, Existence of solutions to the Orlicz-Minkowski problem, Adv. Math. {\bf 344} (2019), 262--288.
		\bibitem{O07}
		O. Johnson\ and\ C. Vignat, Some results concerning maximum R\'{e}nyi entropy distributions, Ann. Inst. H. Poincar\'{e} Probab. Statist. {\bf 43} (2007), no.~3, 339--351.
		\bibitem{K09}
		C.~P. Kitsos\ and\ N.~K. Tavoularis, Logarithmic Sobolev inequalities for information measures, IEEE Trans. Inform. Theory {\bf 55} (2009), no.~6, 2554--2561.
		\bibitem{KL}
		L. Kryvonos\ and\ D. Langharst, Weighted Minkowski's existence theorem and projection bodies, Trans. Amer. Math. Soc. {\bf 376} (2023), no.~12, 8447--8493.	
		\bibitem{LLT}
		D. Langharst, J. Liu\ and\ S. Tang, Weighted $L_p$ Minkowski problem, In progress, (2024).
		\bibitem{Li89}
		Y.~Y. Li, Degree theory for second order nonlinear elliptic operators and its applications, Comm. Partial Differential Equations {\bf 14} (1989), no.~11, 1541--1578.
		\bibitem{L22}
		J. Liu, The $L_p$-Gaussian Minkowski problem, Calc. Var. Partial Differential Equations {\bf 61} (2022), no.~1, Paper No. 28, 23 pp.
		\bibitem{L93}
		E. Lutwak, The Brunn-Minkowski-Firey theory. I. Mixed volumes and the Minkowski problem, J. Differential Geom. {\bf 38} (1993), no.~1, 131--150.
		
		
		
		\bibitem{LLYZ}
		E. Lutwak, S. Lv, D. Yang, G. Zhang, Extensions of Fisher information and Stam's inequality, IEEE Trans. Inform. Theory {\bf 58} (2012), no.~3, 1319--1327.
		\bibitem{LLYZ13}
		E. Lutwak, S. Lv, D. Yang, G. Zhang, Affine moments of a random vector, IEEE Trans. Inform. Theory {\bf 59} (2013), no.~9, 5592--5599.
		\bibitem{LYZ00}
		E. Lutwak, D. Yang\ and\ G. Zhang, $L_p$ affine isoperimetric inequalities, J. Differential Geom. {\bf 56} (2000), no.~1, 111--132.
		\bibitem{LYZ02}
		E. Lutwak, D. Yang\ and\ G. Zhang, Sharp affine $L_p$ Sobolev inequalities, J. Differential Geom. {\bf 62} (2002), no.~1, 17--38.
		\bibitem{LYZ04}
		E. Lutwak, D. Yang\ and\ G. Zhang, Moment-entropy inequalities, Ann. Probab. {\bf 32} (2004), no.~1B, 757--774.
		\bibitem{LYZ05}
		E. Lutwak, D. Yang\ and\ G. Zhang, Cram\'{e}r-Rao and moment-entropy inequalities for Renyi entropy and generalized Fisher information, IEEE Trans. Inform. Theory {\bf 51} (2005), no.~2, 473--478.
		\bibitem{LYZ06}
		E. Lutwak, D. Yang\ and\ G. Zhang, Optimal Sobolev norms and the $L^p$ Minkowski problem, Int. Math. Res. Not. {\bf 2006}, Art. ID 62987, 21 pp.
		\bibitem{LYZ07}
		E. Lutwak, D. Yang\ and\ G. Zhang, Moment-entropy inequalities for a random vector, IEEE Trans. Inform. Theory {\bf 53} (2007), no.~4, 1603--1607.
		\bibitem{Lv23}
		S. Lv, Gradual improvement of the $L_p$ moment-entropy inequality, J. Math. Anal. Appl. {\bf 526} (2023), no.~1, Paper No. 127210, 11 pp.
		\bibitem{MM}
		M. Madiman, J. Melbourne\ and\ P. Xu, Forward and reverse entropy power inequalities in convex geometry, in {\it Convexity and concentration}, 427--485, IMA Vol. Math. Appl., 161, Springer, New York.
		\bibitem{MR14}
		E. Milman\ and\ L. Rotem, Complemented Brunn-Minkowski inequalities and isoperimetry for homogeneous and non-homogeneous measures, Adv. Math. {\bf 262} (2014), 867--908.
		\bibitem{M03}
		H. Minkowski, Volumen und Oberfl\"{a}che, Math. Ann. {\bf 57} (1903), no.~4, 447--495.
		\bibitem{N03}
		S. Nadarajah, The Kotz-type distribution with applications, Statistics {\bf 37} (2003),	
		\bibitem{Sc14}
		R. Schneider, {\it Convex bodies: the Brunn-Minkowski theory}, second expanded edition, Encyclopedia of Mathematics and its Applications, 151, Cambridge Univ. Press, Cambridge, 2014.
		\bibitem{S48}
		C.~E. Shannon, A mathematical theory of communication, Bell System Tech. J. {\bf 27} (1948), 379--423, 623--656.
		\bibitem{SX22}
		W. Sheng\ and\ K. Xue, Flow by Gauss curvature to the $L_p$- Gaussian Minkowski problem, arXiv:2212.01822.
		\bibitem{S59}
		A.~J. Stam, Some inequalities satisfied by the quantities of information of Fisher and Shannon, Information and Control {\bf 2} (1959), 101--112. no.~4, 341--358.
		\bibitem{T14}
		A. Takatsu, Isoperimetric inequality for radial probability measures on Euclidean spaces, J. Funct. Anal. {\bf 266} (2014), no.~6, 3435--3454.
		\bibitem{WXL19}
		Y. Wu, D. Xi\ and\ G. Leng, On the discrete Orlicz Minkowski problem, Trans. Amer. Math. Soc. {\bf 371} (2019), no.~3, 1795--1814.
		\bibitem{WXL20}
		Y. Wu, D. Xi\ and\ G. Leng, On the discrete Orlicz Minkowski problem II, Geom. Dedicata {\bf 205} (2020), 177--190.
		\bibitem{XJL14}
		D. Xi, H. Jin\ and\ G. Leng, The Orlicz Brunn-Minkowski inequality, Adv. Math. {\bf 260} (2014), 350--374.
		\bibitem{Z14}
		G. Zhu, The logarithmic Minkowski problem for polytopes, Adv. Math. {\bf 262} (2014), 909--931.
		\bibitem{Z15}
		G. Zhu, The $L_p$ Minkowski problem for polytopes for $0<p<1$, J. Funct. Anal. {\bf 269} (2015), no.~4, 1070--1094.
		\bibitem{Z17}
		G. Zhu, The $L_p$ Minkowski problem for polytopes for $p<0$, Indiana Univ. Math. J. {\bf 66} (2017), no.~4, 1333--1350.
		\bibitem{ZX14}
		D. Zou\ and\ G. Xiong, Orlicz-John ellipsoids, Adv. Math. {\bf 265} (2014), 132--168.
	\end{thebibliography}
\end{document}